\documentclass[12pt,letterpaper]{article}

\usepackage{amssymb,amsfonts,amscd,amsthm}
\usepackage[all,arc]{xy}
\usepackage{enumerate, bbm}
\usepackage{mathrsfs}
\usepackage{tikz}
\usepackage[left=.8in,top=.9in,right=.8in,bottom=.8in]{geometry}
\usepackage{mathtools}
\usepackage{hyperref}
\usepackage{cleveref}
\usepackage{color}
\usepackage{graphicx}
\usepackage{enumitem} 
\graphicspath{ {TexImages/} }

\newtheorem{thm}{Theorem}[section]
\newtheorem{cor}[thm]{Corollary}
\newtheorem{prop}[thm]{Proposition}
\newtheorem{lem}[thm]{Lemma}
\newtheorem{claim}[thm]{Claim}

\newtheorem{prob}[thm]{Problem}

\theoremstyle{definition}
\newtheorem{defn}{Definition}

\hypersetup{
	colorlinks,
	citecolor=blue,
	filecolor=blue,
	linkcolor=blue,
	urlcolor=blue,
	linktocpage
}

\newcommand{\R}{\mathbb{R}}

\newcommand{\al}{\alpha}
\newcommand{\be}{\beta}

\newcommand{\sig}{\sigma}

\newcommand{\half}{\frac{1}{2}}
\newcommand{\quart}{\frac{1}{4}}

\newcommand{\sm}{\setminus}
\newcommand{\sub}{\subseteq}

\renewcommand{\c}[1]{\mathcal{#1}}

\renewcommand{\SS}[1]{\textcolor{red}{#1}}

\newcommand{\dist}{\mathrm{dist}}
\newcommand{\score}{\mathrm{score}}
\newcommand{\tour}{\vec{S}}

\newcommand{\res}{\mathrm{res}}

\title{Metric Dimensions of March Madness Brackets}
\author{Sam Spiro\footnote{Dept.\ of Mathematics and Statistics, Georgia State University.  Email: sspiro@gsu.edu}}
\date{\today}

\begin{document}

\maketitle

\vspace{-3em}

\begin{abstract}
	Say you and some friends decide to make brackets for March Madness and are told how each of your brackets scored.  The question we ask is: when can you determine how the actual tournament went given your scores?  We determine the exact minimum number of brackets needed to do this for any March Madness-style tournament regardless of the scoring system used, and more generally we prove effective bounds for the problem for arbitrary single-elimination tournaments.
\end{abstract}

\section{Introduction}

\textbf{March Madness}. One of the biggest events in American sports is \textit{March Madness}, which is a single-elimination basketball tournament played by the top college teams in March.  A popular component of this event is fans filling out brackets of predictions for which teams they think will win each match of the tournament.  The ideal goal is to make a bracket where every prediction is correct.  However, given that there are 63 matches that need to be guessed correctly, it is unsurprising that no one has ever achieved this $2^{-63}$ probability event despite tens of millions of brackets being made each year and considerable cash prizes being offered for anyone who could achieve this feat.

While it is essentially impossible to construct a perfect bracket, one can still try and make one which is as accurate as possible and, more importantly, which is more accurate than that of any of your friends.  To this end, it is common for people to form groups (known as pools) where each participant makes a bracket of predictions.  At the end of the tournament, each bracket is given a score based on how accurate their predictions were, and the participant with the highest score is declared the winner of the pool. There is no singular standard way for how brackets are scored, though two common scoring systems are to award some scalar multiple of either $i$ or $2^{i}$ points for each correct guess made in round $i$.

In this paper we will unfortunately not be able to answer the question of how to make a bracket which scores higher than any of your friends.  Instead we ask a different question: if you are given some set of brackets from a pool along with the scores that each of these brackets received, under what conditions can you reconstruct how the tournament actually went?  For example, if one of your friends makes a bracket which scores 0 points, then you can immediately determine from this who wins every match in round 1 (namely, every team which your unlucky friend said would lose in round 1), but you can not say anything about what happens in any of the later rounds of the tournament from this information alone.

\textbf{Metric Dimensions}.  This problem, while whimsical on the surface, lies within the well-studied field of metric dimensions.  Given a metric space $(X,d)$, we say that a set of points $\{x_1,\ldots,x_t\}\sub X$ is a \textit{resolving set} if every $x\in X$ can be uniquely identified by the sequences of distances $(d(x,x_1),\ldots,d(x,x_t))$, and the \textit{metric dimension} of the metric space is the smallest size of a resolving set.  In this language, our problem can be restated as asking when a set of brackets is a resolving set for the metric space $(X,d)$ where $X$ is the set of all brackets and where (informally) $d(B,B')$ measures how much the predictions made by brackets $B,B'$ differ from each other in terms of the number of points that they score.

There is a lot of work in the literature dedicated to metric dimensions and resolving sets.  One of the first problems in this field is a coin-weighing problem due to Erd\H{o}s and Renyi from 1963 \cite{erdHos1963two} which received a fair amount of attention \cite{guy1995coin,lindstrom1964combinatory,soderberg1963combinatory}.  However, the most popular area of study for metric dimensions is the setting where $X$ is the vertex set of a graph $G$ and $d(u,v)$ is the distance between the vertices $u,v$ in $G$.  This graphical setting was introduced independently by Slater~\cite{slater1975leaves} and by Harary and Melter~\cite{harary1976metric}, and since then there has been a tremendous amount of work on these problems, in part due to the many applications of metric dimensions to real-world problems from fields such as chemistry~\cite{chartrand2000resolvability}, robotics~\cite{khuller1996landmarks}, and network analysis~\cite{beerliova2005network}.  We refer the reader to the surveys \cite{kuziak2021metric,tillquist2023getting} for a more comprehensive overview of metric dimensions and their applications.

\textbf{Other Tournaments}.  The most common single-elimination tournaments are those like March Madness which have some $2^r$ teams with each team playing the same number of rounds, but there are other kinds that can be considered.  Indeed, many single-elimination tournaments rely on bye matches when the number of teams is not a power of two, creating asymmetries between the participants.  Some single-elimination tournaments even purposelessly create biases for certain teams, with notable examples being the tournaments appearing in the series \textit{Yu Yu Hakusho} and \textit{Hunter X Hunter} by Yoshihiro Togashi.  One can also consider tournaments where matches are played by more than two teams, which is typical for games such as Mahjong or Texas hold'em.   To capture all of these kinds of tournaments, we need to make some formal definitions. 

\subsection{Definitions}

Our definition of single-elimination tournaments is expressed in the language of digraphs.  For this we recall that given a digraph $\tour$ and a vertex $v$ of this digraph, $N^+(v)$ denotes the set of vertices $w$ such that $v\to w$ is a directed edge of $\tour$, and $N^-(v)$ denotes the set of vertices $u$ such that $u\to v$ is a directed edge.  We say that $v$ is a sink if $N^+(v)=\emptyset$ and that $v$ is a source if $N^-(v)=\emptyset$.

\begin{defn}
	A \textit{single-elimination tournament}\footnote{We will sometimes colloquially refer to single-elimination tournaments simply as ``tournaments'', though we emphasize that this is unrelated to the more usual notion of a tournament in the study of digraphs.} $\tour$ is a directed graph on a finite number of vertices with the following properties: 
	\begin{itemize}
		\item[(a)] $\tour$ contains exactly one sink,
		\item[(b)] $|N^+(v)|=1$ for every vertex $v$ of $\tour$ which is not a sink, 
		\item[(c)] $\tour$ contains no directed cycles, i.e.\ no sequence of distinct vertices $(v_1,\ldots,v_t)$ with $t\ge 2$, $v_1=v_t$, and $v_{i+1}\in N^+(v_i)$ for all $1\le i<t$; and
		\item[(d)] $|N^-(v)|\ne 1$ for every vertex $v$ of $\tour$.
	\end{itemize}
\end{defn}

\begin{figure}
	\hspace{4em}
	\begin{tikzpicture}[scale=1.2]
		
		\newcommand{\drawTree}[6]{
			
			\node (A) at (#1, #2) [circle, draw, inner sep=2pt] {#3};
			\node (B) at (#1 - 1, #2 - 1) [circle, draw, inner sep=2pt] {#4};
			\node (C) at (#1 + 1, #2 - 1) [circle, draw, inner sep=2pt] {#5};
			\node (D) at (#1 - 1.5, #2 - 2) [circle, draw, inner sep=2pt] {a};
			\node (E) at (#1 - 0.5, #2 - 2) [circle, draw, inner sep=2pt] {b};
			\node (G) at (#1 + 0.5, #2 - 2) [circle, draw, inner sep=2pt] {c};
			\node (F) at (#1 + 1.5, #2 - 2) [circle, draw, inner sep=2pt] {d};
			
			\draw[->, thick] (B) -- (A);
			\draw[->, thick] (C) -- (A);
			\draw[->, thick] (D) -- (B);
			\draw[->, thick] (E) -- (B);
			\draw[->, thick] (G) -- (C);
			\draw[->, thick] (F) -- (C);
			
			\node at (#1, #2 - 2.5) {#6};
		}
		
		\drawTree{1}{0}{$z$}{$x$}{$y$}{$\tour$}
		\drawTree{5}{0}{$a$}{$a$}{$c$}{$B_1$}
		\drawTree{9}{0}{$b$}{$b$}{$c$}{$B_2$}
	\end{tikzpicture}
	\caption{A standard single-elimination tournament $\tour$ with 4 players $a,b,c,d$, along with two brackets $B_1,B_2$.  Here $\score_\sig(B_1,B_2)=\sig(y)$ for any scoring system $\sig$ since the only match which $B_1,B_2$ agree on is $y$.  Moreover, one can check that $\{B_1,B_2\}$ is a $\sig$-resolving set for every $\sig$.}\label{figure}
\end{figure}
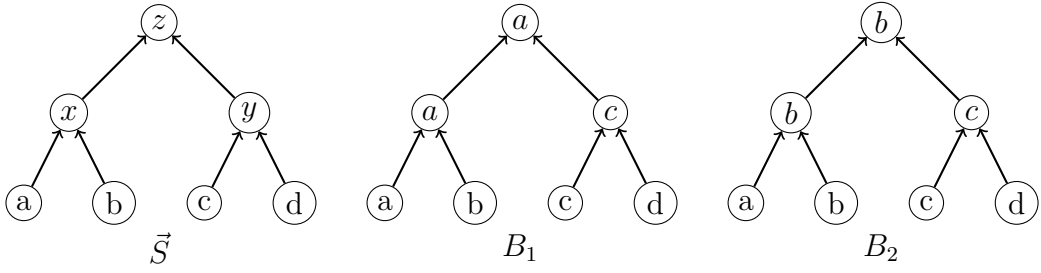

Intuitively properties (a) and (b) say that in a single-elimination tournament, every match except for the finals (represented by the unique sink) has exactly one of its players moving forward in the tournament.  Properties (a), (b), and (c) together are equivalent to saying $\tour$ is an \textit{anti-arborescence}, which is more typically defined as a digraph obtained by starting with an undirected tree and then orienting all of its edges towards some given vertex.   Property (d) is not entirely necessary, but it corresponds to the fact that there is no reason for single-elimination tournaments to have matches involving only a single player (since such a player is guaranteed to go forward), and us assuming this will greatly simplify our proofs.

 It will be useful to introduce a few more definitions around single-elimination tournaments.

\begin{defn}
	Let $\tour$ be a single-elimination tournament.
	\begin{itemize}
		\item We say that a vertex of $\tour$ is a \textit{player} if it is a source, and we denote the set of players of $\tour$ by $P(\tour)$.
		\item We say that a vertex of $\tour$ is a \textit{match} if it is not a source, and we denote the set of matches of $\tour$ by $M(\tour)=V(\tour)\sm P(\tour)$.
		\item We say that $\tour$ is a \textit{standard} single-elimination tournament if it can be obtained by starting with a complete binary tree and then orienting all of its edges towards its center.
		\end{itemize}
\end{defn}
To help distinguish\footnote{One can remember this with some mnemonics: $a,b,c$ are common starting letters for actual names (e.g.\ Alice, Bob, Carl).  The letters $x,y,z$ (which are typically used in math to represent unknown variables) represent that it is unknown ahead of time how a match will go, and we use the last letter $z$ exclusively to refer to the final match of some tournament.  The letters $u,v,w$ lie between these two extremes and thus could be either players or matches.} players and matches, we will exclusively use $a,b,c$ to represent players, $x,y,z$ to represent matches with $z$ always referring to the sink of $\tour$, and $u,v,w$ to represent arbitrary vertices of $\tour$ which could be either players or matches.   We now move on to our definition of brackets.  

\begin{defn}
	Given a single-elimination tournament $\tour$, a \textit{bracket} $B$ of $\tour$ is any function of the form $B:V(\tour)\to P(\tour)$ satisfying the following two conditions:
	\begin{itemize}
		\item[(a)] If $a\in P(\tour)$ is a player, then $B(a)=a$.
		\item[(b)] If $x\in M(\tour)$ is a match, then $B(x)\in \{B(u):u\in N^-(x)\}$.
	\end{itemize}
\end{defn}
Property (b) essentially says that if a bracket predicts player $a$ wins match $x$, then $x$ must either be the first match that $a$ plays in (i.e.\ $a\in N^-(x)$) or the bracket must predict that $a$ wins some earlier match $y\in N^-(x)$.  The last definitions we need are those around scoring systems.
\begin{defn}
	Given a single-elimination tournament $\tour$, a \textit{scoring system} $\sig$ is any function of the form $\sig:M(\tour)\to \R_{>0}$.  We also will make use of the following related definitions:
	\begin{itemize}
		\item Given two brackets $B,B'$ of $\tour$, we define their \textit{$\sig$-score} by
		\[\score_\sig(B,B')=\sum_{x\in M(\tour): B(x)=B'(x)}\sig(x).\]
		\item We say that a set of brackets $\c{B}=\{B_1,\ldots,B_t\}$ is a \textit{$\sig$-resolving set} if for all pairs of distinct brackets $B,B'$, there exists $B_i\in \c{B}$ with $\score_\sig(B_i,B)\ne \score_\sig(B_i,B')$.
		\item We let $\dim(\tour,\sig)$ denote the minimum size of a $\sig$-resolving set and refer to this quantity as the \textit{metric dimension} of $(\tour,\sig)$.
	\end{itemize}
\end{defn}
We note that the function $d(B,B'):=\sum_{x\in M(\tour)} \sig(x)-\score_\sig(B,B')$ defines a metric, and hence a resolving set with respect to  $\score_\sig$ is the same as a resolving set with respect to this metric.

\subsection{Main Results}
\textbf{Metric Dimensions}.  The first question we tackle is: given a single-elimination tournament $\tour$ and a scoring system $\sig$, what is $\dim(\tour,\sig)$, the smallest size of a $\sig$-resolving set?

To start, we consider \textit{standard} single-elimination tournaments $\tour$ like those in March Madness.  If such a tournament has $n$ players, then apriori $\dim(\tour,\sig)$ could be as large as the number of brackets $2^{n-1}$.  Perhaps surprisingly, it turns out that $\dim(\tour,\sig)$ is substantially smaller than this and, even more so, is independent of the choice of the scoring system $\sig$.

\begin{thm}\label{standard}
	If $\tour$ is a standard single-elimination tournament on $n$ players with $n\ge 2$ a power of 2, then for every scoring system $\sig$ we have
	\[\dim(\tour,\sig)=\frac{n}{2}.\]
	Moreover, for all $n$ there exists a set of $n/2$ brackets $\c{B}$ which is $\sig$-resolving for every $\sig$.
\end{thm}

Theorem~\ref{standard} follow from more general upper and lower bounds on $\dim(\tour,\sig)$ for arbitrary choices of $\tour$.  Our strongest upper bound Theorem~\ref{upper bound technical} is too technical to state here, so we instead record one of its corollaries which says that $\dim(\tour,\sig)$ is always strictly less than the number of players for every choice of $\tour$ and $\sig$.

\begin{thm}\label{upper bound}
	For every single-elimination tournament $\tour$ with $n$ players and any scoring system $\sig$, we have
	\[\dim(\tour,\sig)\le n-1.\]
\end{thm}

This bound is best possible for arbitrary $\tour$, as can be seen by considering any tournament with only 1 match.  Our general lower bound requires a definition to state, but has the advantage of being tight for \textit{every} single-elimination tournament $\tour$ for at least one choice of $\sig$.

\begin{defn}\label{Definition Pu}
	A sequence of vertices $(u_1,\ldots,u_t)$ of a digraph $\tour$ is a \textit{directed walk from $u_1$ to $u_t$} if $u_{i+1}\in N^+(u_i)$ for all $1\le i<t$.  For a single-elimination tournament $\tour$, we define the \textit{player set} $P(u)$ for each $u\in V(\tour)$ to be the set of players $a$ for which there exists a directed walk from $a$ to $u$.
\end{defn}

For example, $P(a)=\{a\}$ for every player $a$ since players are sources and hence have no non-trivial walks to them, and intuitively one should think of $P(x)$ for a match $x$ as being the set of players which could theoretically compete in $x$.
\begin{thm}\label{lower bound}
	For every single-elimination tournament $\tour$ with at least 2 players and any scoring system $\sig$, we have
	\[\dim(\tour,\sig)\ge \max_{x\in M(\tour)}\left( |P(x)|-\max_{u\in N^-(x)} |P(u)|\right).\]
	Moreover, for every $\tour$ there exists a scoring system $\sig$ such that this bound is tight.
\end{thm}

To get a handle on this quantity, we observe that if the sink $z$ of $\tour$ has $N^-(z)=\{u,v\}$ with $|P(u)|=|P(v)|=n/2$ (i.e.\ if the ``left'' half of the bracket and the ``right'' half of the bracket have the same number of players), then Theorem~\ref{lower bound} applied with $x=z$ gives $\dim(\tour,\sig)\ge n/2$.  Thus for most reasonable tournaments, Theorem~\ref{lower bound} shows that $\dim(\tour,\sig)$ is linear in the number of players for arbitrary scoring systems $\sig$.  We emphasize, however, that the bound of Theorem~\ref{lower bound} can be as small as 1 for arbitrarily large tournaments and that this lower bound is tight for appropriate choices of $\sig$ in these cases.  We also emphasize that there exist single-elimination tournaments such that $\dim(\tour,\sig)$ depends on $\sig$, meaning that Theorem~\ref{lower bound} will not be optimal for every choice of scoring system in general. 

\textbf{Resolving Numbers}.  Finding resolving sets of size exactly $\dim(\tour,\sig)$ can be difficult.  What if we wanted to be lazy and just take a large number of distinct brackets and hope that our set ends up being resolving?  At what point could we guarantee that our collection is always resolving?

To this end, we define the \textit{resolving number} $\res(\tour,\sig)$ to be the minimum number $r$ such that every set of $r$ brackets is a $\sig$-resolving set for $\tour$.  This definition is analogous to the notion of the resolving number of graphs introduced by Chartrand, Poisson, and Zhang in 2000~\cite{chartrand2000resolvability} which has received a fair amount of attention within the metric dimension community \cite{garijo2013metric,garijo2013resolving,tennenhouse2015new}.  

We trivially have $\dim(\tour,\sig)\le \res(\tour,\sig)$, and as such Theorem~\ref{lower bound} suggests that $\res(\tour,\sig)$ will often be at least linear in the number of players.  In fact, $\res(\tour,\sig)$ turns out to be nearly equal to the total number of brackets for almost every $\tour$ and is largely controlled by the maximum probability of a given player winning a uniform random bracket.

\begin{thm}\label{resolving with qmax alone}
	Let $\tour$ be a single-elimination tournament with at least 2 players.   Let $z$ denote the sink of $\tour$, let $R$ be a uniformly random bracket of $\tour$, and define
	\[q_{\max}=\max_{a\in P(\tour)}\Pr[R(z)=a].\]
	If there are a total of $N$ brackets for $\tour$, then for any scoring system $\sig$ we have
	\[(1-2q_{\mathrm{max}})N< \res(\tour,\sig)\le (1-q_{\max})N.\]
\end{thm}
For example, if $\tour$ is a standard single-elimination tournament with $n$ players then $q_{\max}=n^{-1}$ since each player is equally likely to win the finals, so Theorem~\ref{resolving with qmax alone} gives bounds of the form $(1-2n^{-1})2^{n-1}\le \res(\tour,\sig)\le (1-n^{-1})2^{n-1}$.  We can say some slightly stronger facts about $\res(\tour,\sig)$ in general; see the Appendix for more on this.


\textbf{Organization}.  We start with a high-level proof sketch of Theorem~\ref{standard} in Section~\ref{sec:sketch}.  We then establish some general preliminary results in Section~\ref{sec:preliminaries}, after which we prove all of our lower bounds  in Section~\ref{sec:lower} and all of our upper bounds in Section~\ref{sec:upper}.

\section{Proof Sketch}\label{sec:sketch}
To illustrate our key proof ideas, we give a brief proof sketch of Theorem~\ref{standard}.  To this end, let $\tour$ denote a standard single-elimination tournament with $n\ge 2$ a power of 2 number of players.

To see that $\dim(\tour,\sig)\ge n/2$, observe that any set of brackets $\c{B}$ of size less than $n/2$ has some player $a$ from the ``left side'' of $\tour$ which is never predicted to win the final match $z$ by any bracket in $\c{B}$, and similarly there exists some player $b$ from the ``right side'' which is never predicted to win $z$.  Thus, if we let $B,B'$ be two brackets which agree with each other everywhere outside of $z$ with $B$ predicting that $a$ wins and $B'$ predicting that $b$ wins, then no bracket in $\c{B}$ can distinguish between $B$ and $B'$ (since a bracket can distinguish these two if and only if the bracket has $a$ or $b$ winning the final match).  This proves the lower bound of $n/2$.  More generally, we can replicate this same argument for any single-elimination tournament and any match in place of $z$ to give our key lemma Proposition~\ref{necessary condition} which is the crux of every single one of our lower bounds.

We prove the upper bound $\dim(\tour,\sig)\le n/2$ by induction on $n$.  To this end, we observe that the ``right half'' of $\tour$ is essentially a standard single-elimination tournament with $n/2$ players, and as such we can inductively find a set of brackets $\widetilde{B}_1,\ldots,\widetilde{B}_{n/4}$ which is resolving for any scoring system on this smaller tournament.  Now label the players on the ``left half'' of $\tour$ by $x_1,x_2,\ldots,x_{n/4},x'_1,\ldots,x'_{n/4}$ in such a way that $x_i,x'_i$ play against each other in their first match.  For each $1\le i\le n/4$, define $B_i$ to be any bracket for $\tour$ which has $x_i$ winning the whole tournament and which equals $\widetilde{B}_i$ when restricted to the ``right half'' of $\tour$, and further define $B'_i$ to be identical to $B_i$ except that it has $x'_i$ winning the tournament.  Note that the brackets given in Figure~\ref{figure} are exactly of this form.

Crucially, we observe that for every bracket $B$, the difference $\score_\sig(B_i,B)-\score_\sig(B_i',B)$ determines the set of matches that $x_i$ and $x'_i$ win in $B$.  For example, this difference is positive if and only if $x_i$ wins the match between $x_i$ and $x'_i$ in round 1, with the exact difference equaling the sum of all the scores of the matches that $x_i$ wins in $B$.  As such, knowing $\score_\sig(B_i,B)$ and $\score_\sig(B_i',B)$ for all $i$ completely determines the matches won by each of the $x_i$ and $x'_i$ players.  Because of this knowledge of the ``left half'' of the tournament, we can write $\score_\sig(B_i,B)$ as a function of the score restricted only to the ``right half'' of $\tour$.  Because the restrictions of these $B_i$ brackets are exactly the resolving set $\widetilde{B}_1,\ldots,\widetilde{B}_{n/4}$, we can in turn completely determine the ``left half'' as well, proving that this is indeed a resolving set.  This proves our upper bound for standard single-elimination tournaments, and our upper bounds for other tournaments similarly rely on ``lifting'' resolving sets from sub-tournaments together with using pairs of brackets $B_i,B'_i$ which differ by each other on a single player.

\section{Preliminaries}\label{sec:preliminaries}
\textbf{Player Sets}.  Our first goal will be to establish some basic facts about the player sets $P(u)$ from Definition~\ref{Definition Pu} through our forthcoming Proposition~\ref{Pu Sets}. For this, we need some lemmas about single-elimination tournaments which will not appear outside this section of the paper.

\begin{lem}\label{sources and sinks}
	Let $\tour$ be any finite digraph without directed cycles.
	\begin{itemize}
		\item[(i)] There exist no closed walks in $\tour$, i.e.\ no directed walks $(u_1,\ldots,u_t)$ with $u_1=u_t$ and $t\ge 2$.
		\item[(ii)] For every vertex $u$ of $\tour$, there exists a directed walk from $u$ to a sink.
		\item[(iii)] For every vertex $u$ of $\tour$, there exists a directed walk from a source to $u$.
	\end{itemize}
\end{lem}
\begin{proof}
	For (i), assume for contradiction that there exists a closed walk $(u_1,\ldots,u_t)$ and choose such a walk with $t$ as small as possible.  Note that $u_i\ne u_j$ for all $1\le i<j<t$, as otherwise $(u_i,\ldots,u_j)$ would be a shorter closed walk.  But this means $(u_1,\ldots,u_t)$ is a directed cycle, a contradiction.
	
	For (ii), let $u$ be an arbitrary vertex.  Let $u_1=u$, and iteratively as long as $N^+(u_i)\ne \emptyset$ we define $u_{i+1}$ to be an arbitrary vertex of $N^+(u_i)$.  If this process terminates at some $u_t$ with $N^+(u_t)=\emptyset$, then by definition $u_t$ is a sink and $(u_1,\ldots,u_t)$ is a directed walk from $u$ to this sink.  Otherwise, because $\tour$ is a finite digraph there must exist some indices $i<j$ in the infinite sequence $u_1,u_2,\ldots$ with $u_i=u_j$.  This means that $(u_i,u_{i+1},\ldots,u_j)$ is a closed walk of $\tour$, a contradiction to (i).  We conclude that $u$ has a walk to a sink.  The proof of (iii) is entirely analogous but with $N^+$ replaced by $N^-$; we omit the details.
\end{proof}
\begin{lem}\label{prefix path}
	If $\tour$ is a single-elimination tournament and if $(u_1,\ldots,u_s)$ and $(u_1',\ldots,u_t')$ are two directed walks from the same vertex $u=u_1=u_1'$, then one of these walks is a prefix for the other, i.e.\ either $(u_1,\ldots,u_s)=(u'_1,\ldots,u'_s)$ or $(u'_1,\ldots,u'_t)=(u_1,\ldots,u_t)$.
\end{lem}
\begin{proof}
	Let $i$ be the largest integer such that $u_i=u'_i$, noting that such a (largest) integer must exist since $u_1=u'_1$.  If $i<\min\{s,t\}$ then by definition $u_{i+1},u'_{i+1}$ are two distinct vertices which are in $N^+(u_i)=N^+(u'_i)$, a contradiction to single-elimination tournaments having at most one out-neighbor per vertex.  We thus must have, say, $i=s\le t$ which means that $(u_1,\ldots,u_s)=(u'_1,\ldots,u'_s)$, proving the result.
\end{proof}

We can now state our main properties of the player sets $P(u)$, which we recall are defined to be the set of players $a$ which have a directed walk to $u$.
\begin{prop}\label{Pu Sets}
	Let $\tour$ be a single-elimination tournament.
	\begin{itemize}
		\item[(i)] Every $u\in V(\tour)$ has $P(u)\ne \emptyset$.
		\item[(ii)] For all $u,v\in V(\tour)$ we have $P(u)= P(v)$ if and only if $u= v$
		\item[(iii)] We have $P(u)\sub P(v)$ for $u,v\in V(\tour)$ if and only if there exists a directed walk from $u$ to $v$.
		\item[(iv)] For all distinct $u,v\in V(\tour)$ we either have $P(u)\cap P(v)=\emptyset,\ P(u)\subsetneq P(v)$, or $P(v)\subsetneq P(u)$.
		\item[(v)] For every match $x\in M(\tour)$ the sets $\{P(u):u\in N^-(x)\}$ partition $P(x)$ and satisfy that $P(u)\subsetneq P(x)$ and $P(u)\ne \emptyset$ for all $u\in N^-(x)$.
	\end{itemize}
\end{prop}
We note that (ii)--(iv) imply that one can define a partial order on $V(\tour)$ by having $u<v$ whenever $P(u)\subsetneq P(v)$.
\begin{proof}
	For (i), we have by Lemma~\ref{sources and sinks} that each $u$ has a directed walk from a source, i.e.\ from a player, which by definition means this player belongs to $P(u)$.
	
	We next prove (v), for which we will need to prove the easy direction of (iii).
	\begin{claim}\label{Claim for Pu}
		If $u,v\in V(\tour)$ are such that there exists a directed walk from $u$ to $v$, then $P(u)\sub P(v)$.
	\end{claim}
	\begin{proof}
		Let $(u_1,\ldots,u_t)$ be a directed walk from $u$ to $v$.  By definition, for each $a\in P(u)$ there exists a directed walk $(u'_1,\ldots,u'_s)$ from $a$ to $u$.  But then $(u'_1,\ldots,u'_s=u_1,u_2,\ldots,u_t)$ is a directed walk from $a$ to $v$, implying $a\in P(v)$ and hence $P(u)\sub P(v)$ since $a\in P(u)$ was arbitrary. 
	\end{proof}
	From this claim, we immediately have that $\bigcup_{u\in N^-(x)} P(u)\sub P(x)$ since each $u\in N^-(x)$ has a directed walk to $x$, namely $(u,x)$.  On the other hand, for each $a\in P(x)$ there exists a directed walk $(u_1,\ldots,u_t)$.  This means there exists a directed walk from $a$ to an in-neighbor of $x$, namely $u_{t-1}\in N^-(x)$, meaning that $a\in P(u)$ for some $u\in N^-(x)$ and hence $P(x)\sub \bigcup_{u\in N^-(x)} P(u)$, proving that these sets are equal.  Now assume for contradiction that there exist distinct $u,u'\in N^-(x)$ with  $P(u)\cap P(u')\ne \emptyset$, say with $a\in P(u)\cap P(u')$.  This implies that there exists a directed walk $(u_1,\ldots,u_s)$ from $a$ to $u$ and a directed walk $(u'_1,\ldots,u'_t)$ from $a$ to $u'$.  By Lemma~\ref{prefix path} we must have, say, $(u_1,\ldots,u_s)=(u'_1,\ldots,u'_s)$ as well as $s<t$ since $u'_s=u_s=u\ne u'=u'_t$.  This in turn means that we must have $u'_{s+1}= x$, as otherwise $x,u'_{s+1}$ would be two distinct vertices in $N^+(u'_s)=N^+(u_s)=N^+(u)$.  This implies that $(u'_{s+1},u'_{s+2},\ldots,u'_t,x)$ is a closed walk in $\tour$, a contradiction to Lemma~\ref{sources and sinks}.  We conclude that the $P(u)$ sets partition $P(x)$.   The fact that these sets are non-empty follows from (i).  To show $P(u)\subsetneq P(x)$ for any $u\in N^-(x)$, we use that there exists some $u'\in N^-(x)\sm \{u\}$ by definition of single-elimination tournaments.  Then by what we have proven up to this point, we find $P(u)\sub P(x)\sm P(u')\subsetneq P(x)$, proving the result.
	
	We now prove (ii).  Certainly $P(u)=P(v)$ whenever $u=v$.  Assume now that $P(u)=P(v)$.  By (i) there is some $a\in P(u)\cap P(v)$, and again by Lemma~\ref{prefix path} this means that there exists some directed walk $(u_1,\ldots,u_t)$ from $a$ to either $u$ or $v$ with both $u,v$ appearing somewhere in this walk.  Without loss of generality assume $u_t=u$.  If $v=u$ then we are done, otherwise $v=u_i$ for some $i<t$.  By (v) it follows that $P(v)\subsetneq P(u_{i+1})\subsetneq P(u_{i+2})\cdots \subsetneq P(u_t)=P(u)$, a contradiction to our assumption of $P(u)=P(v)$.
	
	For (iii), the first half follows from Claim~\ref{Claim for Pu}, so assume now that $P(u)\sub P(v)$.  The result is trivial if $u=v$ so we may assume $u\ne v$ and hence that $P(u)\subsetneq P(v)$ by (ii).  By (i), there exists some $a\in P(u)=P(u)\cap P(v)$.  By Lemma~\ref{prefix path} this means there exists some directed walk $(u_1,\ldots,u_t)$ which has $u_t\in \{u,v\}$ and $u_i\in \{u,v\}$ for some $i<t$, where here we used that $u\ne v$ to conclude $i<t$.  If $u_t=u$ then this means $(u_i,\ldots,u_t)$ is a directed walk from $v$ to $u$, meaning that $P(v)\sub P(u)$ by Claim~\ref{Claim for Pu}, a contradiction to $P(u)\subsetneq P(v)$.  Thus we must have $u_t=v$, meaning that $(u_i,\ldots,u_t)$ is a directed walk from $u$ to $v$ as desired.
	
	For (iv), if $P(u)\cap P(v)=\emptyset$ then there is nothing to prove, so assume there exists $a\in P(u)\cap P(v)$.  By Lemma~\ref{prefix path} there must exist a directed walk $(u_1,\ldots,u_t)$ from $a$ to one of $u,v$ with the other vertex appearing somewhere in this walk, say without loss of generality that $u_t=u$ and $u_i=v$, noting that $i<t$ by assumption of $u,v$ being distinct.  This implies that $P(u)\sub P(v)$ by (iii) and this inclusion must be strict by (ii), proving the result.
\end{proof}

\textbf{Brackets}.   We now prove lemmas about brackets, starting with some basic observations.

\begin{lem}\label{bracket facts}
	Let $\tour$ be a single-elimination tournament and $B$ a bracket.
	\begin{itemize}
		\item[(i)] We have $B(u)\in P(u)$ for all $u \in V(\tour)$.
		\item[(ii)] For each $v\in V(\tour)$, we have $B(u)=B(v)$ for every $u\in V(\tour)$ with $P(u)\sub P(v)$ and $B(v)\in P(u)$.
	\end{itemize}
\end{lem}
\begin{proof}
	For (i), assume this was false for some $u$ and choose such a $u$ with $|P(u)|$ as small as possible.  If $u$ is a source then $B(u)=u\in P(u)$ by definition of brackets, so we can assume $u$ is not a source.  Note that $P(u')\subsetneq P(u)$ for all $u'\in N^-(u)$ by Proposition~\ref{Pu Sets}(v), so by definition of brackets and the minimality of $u$ we have that
	\[B(u)\in \{B(u'):u'\in N^-(u)\}\sub \bigcup_{u'\in N^-(u)} P(u')=P(u),\]
	proving the result.
	
	For (ii), assume for contradiction this was false for some $u,v$ and choose such a $u$ with $|P(u)|$ as large as possible.  Because $P(u)\sub P(v)$, by Proposition~\ref{Pu Sets}(iii) there exists some directed walk $(u_1,\ldots,u_t)$ from $u$ to $v$ with $t\ge 2$ since the statement is vacuously true for $u=v$.  Because $u_2$ also has a directed walk to $v$ we have $P(u_2)\sub P(v)$ and also $P(u)\subsetneq P(u_2)$ by Proposition~\ref{Pu Sets}(iii) and (v), so by the maximality of $u$ we have $B(u_2)=B(v):=a$.  By (i) of this lemma and the fact that $B(u_2)\in \{B(u'):u'\in N^-(u_2)\}$ for brackets implies that $B(u_2)=B(u')=a$ for some $u'\in N^-(u_2)$ with $a\in P(u')$.  But $a\in P(u)$ by hypothesis, so Proposition~\ref{Pu Sets}(v) implies that we must have $u'=u$ and hence that $B(u)=a$, a contradiction.
\end{proof}

We end with a lemma which allows us to modify brackets in an appropriate way.
\begin{lem}\label{bracket modification}
	Let $\tour$ be a single-elimination tournament.  If $\widehat{B}$ is a bracket and $a\in P(\tour)$, then the function $\widehat{B}_a:V(\tour)\to P(\tour)$ defined by  $\widehat{B}_a(u)=a$ whenever $a\in P(u)$ and $\widehat{B}_a(u)=\widehat{B}(u)$ otherwise is a bracket.
\end{lem}
That is, if one takes a bracket and changes it so that a given player wins all of their matches, then this is still a bracket.
\begin{proof}
	Observe that $\widehat{B}_a(a)=a$ and $\widehat{B}_a(b)=\widehat{B}(b)=b$ for any other $b\in P(\tour)$ because $\widehat{B}$ is a bracket and $P(b)=\{b\}$ for players $b$.  It thus remains to check that for every match $x$, we have $\widehat{B}_a(x)\in \{\widehat{B}_a(u):u\in N^-(x)\}$.  If $a\notin P(x)$ then $a\notin P(u)$ for any $u\in N^-(x)$ by Proposition~\ref{Pu Sets}(v), and hence
	\[\widehat{B}_a(x)=\widehat{B}(x)\in \{\widehat{B}(u):u\in N^-(x)\}=\{\widehat{B}_a(u):u\in N^-(x)\}.\]
	Otherwise if $a\in P(x)$ then there exists some $u'\in N^-(x)$ with $a\in P(u')$, so
	\[\widehat{B}_a(x)=a=\widehat{B}_a(u')\in \{\widehat{B}_a(u):u\in N^-(x)\}.\]
	We conclude that $\widehat{B}_a$ is a bracket.
\end{proof}
We emphasize that throughout this paper, whenever we use the notation $\widehat{B}$ we do so with the intention of eventually applying Lemma~\ref{bracket modification} to this bracket.

\section{Lower Bounds}\label{sec:lower}

To prove our general lower bounds we need to show the existence of some special matches.  Given two distinct players $a,b$, we intuitively want to define a match $x_{a,b}$ to be the unique match of $\tour$ where $a$ and $b$ face off against each other.  For this we need the following technical statement.

\begin{lem}\label{matches have walks}
	Let $\tour$ be a single-elimination tournament, $a,b$ distinct players, and $u,x\in V(\tour)$.  If $a,b\in P(u)\cap P(x)$ and if there exist $u_a,u_b\in N^-(x)$ with $ P(u_a)\cap \{a,b\}=\{a\}$ and $P(u_b)\cap \{a,b\}=\{b\}$, then there exists a directed walk from $x$ to $u$. 
\end{lem}

\begin{proof}
	If $u=x$ then there is nothing to prove, so assume this is not the case.  By Proposition~\ref{Pu Sets}(iv) we must have either $P(u)\subsetneq P(x)$ or $P(x)\subsetneq P(u)$.  This latter case implies the result by Proposition~\ref{Pu Sets}(iii), so we may assume  $P(u)\subsetneq P(x)$ which by Proposition~\ref{Pu Sets}(iii) implies there exists a directed walk $(u_1,\ldots,u_t)$ from $u$ to $x$ with $t\ge 2$ since $u\ne x$.  Because $(u_1,\ldots,u_{t-1})$ is a directed walk from $u$ to $u_{t-1}$ we have $a,b\in P(u_{t-1})$ by Proposition~\ref{Pu Sets}(iii). By Proposition~\ref{Pu Sets}(v) $u_{t-1}\in N^-(x)$ must be the only in-neighbor of $x$ with $a$ or $b$ in its player set, contradicting the existence of the vertices $u_a,u_b$.
\end{proof}

We now prove the existence of the matches $x_{a,b}$ hinted at above.
\begin{lem}\label{matches exist}
	For any single-elimination tournament $\tour$ and any two distinct players $a,b$, there exists a unique match $x_{a,b}$ which has in-neighbors $u_a,u_b\in N^-(x)$ satisfying $P(u_a)\cap \{a,b\}=\{a\}$ and $P(u_b)\cap \{a,b\}=\{b\}$.
\end{lem}
\begin{proof}
	Define $x_{a,b}$ to be a vertex in $\tour$ with $a,b\in P(x_{a,b})$ which has $|P(x_{a,b})|$ as small as possible subject to this condition.  Note that such a vertex must exist since the sink $z$ satisfies $P(z)=P(\tour)$ due to Lemma~\ref{sources and sinks}(ii) and the fact that single-elimination tournaments have a unique sink.   By Proposition~\ref{Pu Sets}(v), for each player $c\in P(x_{a,b})$ there exists a unique in-neighbor $u_c$ with $c\in P(u_c)$.  If $u_a\ne u_b$ then we are done, otherwise we have $a,b\in P(u_a)$ and $P(u_a)\subsetneq P(x)$ by Proposition~\ref{Pu Sets}(v), a contradiction to our choice of $x_{a,b}$.  We conclude that $x_{a,b}$ is a vertex satisfying the conditions of the lemma.
	
	For uniqueness, assume that there existed two vertices $x_{a,b},x'_{a,b}$ satisfying the conditions of the lemma.  Note that we must have $a,b\in P(x_{a,b})\cap P(x_{a,b}')$ since, for example, $u_a\in N^-(x_{a,b})$ has $a\in P(u_a)\sub P(x_{a,b})$ by Proposition~\ref{Pu Sets}(v).  By Lemma~\ref{matches have walks}, there exists a directed walk $(u_1,\ldots,u_t)$ from $x_{a,b}$ to $x'_{a,b}$ as well as a directed walk $(u'_1,\ldots,u'_s)$ from $x'_{a,b}$ to $x_{a,b}$.  But this means that $(u_1,\ldots,u_t,u_1',\ldots,u'_s)$ is a closed walk, a contradiction to Lemma~\ref{sources and sinks}.

\end{proof}

With this in mind, all of our lower bounds can be stated in terms of the following simple but crucial fact about resolving sets.

\begin{prop}\label{necessary condition}
	Let $\tour$ be a single-elimination tournament. If $\c{B}$ is a set of brackets and if there exist two distinct players $a,b$ such that $B_i(x_{a,b})\notin \{a,b\}$ for every $B_i\in \c{B}$, then $\c{B}$ is not $\sig$-resolving for any scoring system $\sig$.
\end{prop}
\begin{proof}
	The proof is based on an observation.
	\begin{claim}
		It suffices to prove that there exist brackets $B,B'$ with $B(u)=B'(u)$ for all $u\ne x_{a,b}$ and with $B(x_{a,b})=a$ and $B'(x_{a,b})=b$.
	\end{claim}
	\begin{proof}
		Say that such brackets $B,B'$ exist.  For every bracket $\widetilde{B}$, define $\al(\widetilde{B})$ to be the indicator function which is 1 if $\widetilde{B}(x_{a,b})=a$ and which is 0 otherwise.  Similarly define $\be(\widetilde{B})$ to be the indicator function for $\widetilde{B}(x_{a,b})=b$.  We then observe that
		\[\score_\sig(\widetilde{B},B)-\score_\sig(\widetilde{B},B')=\sum_{x:\widetilde{B}(x)=B(x)}\sig(x)-\sum_{x:\widetilde{B}(x)=B'(x)}\sig(x)=(\al(\widetilde{B})-\be(\widetilde{B}))\sig(x_{a,b}),\]
		since a term $\sig(x)$ with $x\ne x_{a,b}$ appears in one sum if and only if it appears in the other sum by assumption of $B(x)=B'(x)$.  In particular, we see that $\score_\sig(\widetilde{B},B)=\score_\sig(\widetilde{B},B')$ whenever $\al(\widetilde{B})=\be(\widetilde{B})=0$, i.e.\ whenever $\widetilde{B}(x_{a,b})\notin \{a,b\}$.  As $\al(\widetilde{B})=\be(\widetilde{B})=0$ holds for every bracket in $\c{B}$ by hypothesis, we conclude that $\c{B}$ is not a $\sig$-distinguishing set for any choice of $\sig$.
	\end{proof}
	
	To show that brackets as in the claim exist, we will start with an arbitrary bracket  $\widehat{B}$ and then modify it according to Lemma~\ref{bracket modification} in an appropriate way.  From now on we fix some arbitrary bracket $\widehat{B}$.
	
	First consider the case that $x_{a,b}$ is the sink of $\tour$.  Define $B:=(\widehat{B}_b)_a$ and $B':=(\widehat{B}_a)_b$; that is, $B$ is defined by changing $\widehat{B}$ so that $b$ wins every possible match and then changing $\widehat{B}_b$ so that $a$ wins every possible match, with $B'$ defined in the same way but with the order of the players reversed.  By construction $B,B'$ can only differ on $u$ with $a,b\in P(u)$.  By Lemma~\ref{matches have walks}, there must exist a directed walk from $x_{a,b}$ to any such $u$, but since $x_{a,b}$ is a sink this is only possible for $u=x_{a,b}$.  In this case we have $B(x_{a,b})=a$ and $B'(x_{a,b})=b$, so $B,B'$ satisfy the condition of the claim, proving the result
	
	Now consider the case that there exists some $y\in N^+(x_{a,b})$.  By definition of single-elimination tournaments, there exists some $v\in N^-(y)\sm \{x\}$ and we let $c\in P(v)$ be an arbitrary player which exists by Proposition~\ref{Pu Sets}(i).  We now define $B=((\widehat{B}_b)_a)_c$ and $B'=((\widehat{B}_a)_b)_c$.  Again the only way we could fail to have $B(u)=B'(u)$ is if $a,b\in P(u)$, but any such vertex is either equal to $x_{a,b}$ or can be reached from $y$ by a directed walk by Lemma~\ref{matches have walks}, and hence has $c\in P(u)$ and $B(u)=c=B'(u)$ by construction.  We also have $B(x_{a,b})=a$ and $B'(x_{a,b})=b$, proving the result.
\end{proof}

This result quickly gives a technical lower bound for $\res(\tour,\sig)$, which we recall is the smallest number $r$ such that every set of $r$ brackets is $\sig$-resolving.
\begin{prop}\label{resolving lower bound technical}
	Let $\tour$ be a single-elimination tournament on at least two players.  Let $R$ be a uniform random bracket of $\tour$ and define 
	\[q_{\mathrm{pair}}=\min_{a,b\in P(\tour),\ a\ne b} \Pr[R(x_{a,b})\in \{a,b\}].\]
	If $\tour$ has $N$ total brackets, then
	\[\res(\tour,\sig)>(1-q_{\mathrm{pair}})N.\] 
\end{prop}
\begin{proof}
	For two distinct players $a,b$, let $\c{B}_{a,b}$ be the set of brackets $B$ with $B(x_{a,b})\notin \{a,b\}$.   By Proposition~\ref{necessary condition}, each of the sets $\c{B}_{a,b}$ fails to be a resolving set of $\tour$ for any choice of $\sig$, and hence
	\[\res(\tour,\sig)>\max_{a,b} |\c{B}_{a,b}|=\max_{a,b} N-\Pr[R(x_{a,b})\in \{a,b\}] N=(1-q_{\mathrm{pair}})N,\]
	where the first equality used that the number of brackets with $B(x_{a,b})\in \{a,b\}$ is exactly $\Pr[R(x_{a,b})\in \{a,b\}] N$ since $R$ is uniformly random.
\end{proof}
This technical bound turns out to imply our simpler lower for $\res(\tour,\sig)$ given in Theorem~\ref{resolving with qmax alone}.

\begin{cor}\label{resolving qpair and qmax}
	Let $\tour$ be a single-elimination tournament on at least two players with source $z$.  Let $R$ be a uniformly random bracket of $\tour$ and define $q_{\max}=\max_{a\in P(\tour)}\Pr[R(z)=a]$.
	If there are a total of $N$ brackets for $\tour$, then for any scoring system $\sig$ we have
	\[\res(\tour,\sig)>(1-2q_{\max})N.\]
\end{cor}
\begin{proof}
	We aim to show that $q_{\mathrm{pair}}\le 2q_{\max}$, from which the result will follow by Proposition~\ref{resolving lower bound technical}.  
	
	We claim that there exist distinct players $a,b$ with  $z=x_{a,b}$.  Indeed, $z$ has at least two in-neighbors $u,u'$ by definition of single-elimination tournaments.  Let $a\in P(u)$ and $b\in P(u')$ be arbitrary players, which exist by Proposition~\ref{Pu Sets}(i). By Proposition~\ref{Pu Sets}(v) we have that $u,u'$ are the unique in-neighbors of $z$ which have $a,b$ in their player sets, so taking $u_a=u$ and $u_b=u'$ verifies that $z$ satisfies the conditions of $x_{a,b}$, proving the claim.
	
	With $a,b$ as above, we find
	\[q_{\mathrm{pair}}\le \Pr[R(z)\in \{a,b\}]= \Pr[R(z)=a]+\Pr[R(z)=b]\le 2 q_{\max},\]
	proving the result.
\end{proof}

We now establish our lower bound for $\dim(\tour,\sig)$ as stated in Theorem~\ref{lower bound}.
\begin{prop}\label{lower bound deterministic}
	For every single-elimination tournament $\tour$ and any scoring system $\sig$, we have
	\[\dim(\tour,\sig)\ge \max_{x\in M(\tour)}\left( |P(x)|-\max_{u\in N^-(x)} |P(u)|\right).\]
\end{prop}
\begin{proof}
	Let $\c{B}$ be an arbitrary $\sig$-resolving set, and assume for contradiction that there exists a match $x\in M(\tour)$ with $|\c{B}|<|P(x)|-\max_{u\in N^-(x)} |P(u)|$.  For each $u\in N^-(x)$, let $\c{B}_u\sub \c{B}$ be the set of brackets $B$ with $B(x)=B(u)$, noting that the $\c{B}_u$ sets partition $\c{B}$ since by definition of being a bracket, $B(x)\in \{B(u):u\in N^-(x)\}$ and we have $B(u)\ne B(u')$ for distinct $u,u'\in N^-(x)$ by Proposition~\ref{Pu Sets}(v) and Lemma~\ref{bracket facts}(i).
	
	We claim that there exist at least two vertices $u\in N^-(x)$ such that $|\c{B}_u|<|P(u)|$.  Indeed, if this inequality held for exactly one vertex $u'$ then
	\[|\c{B}|=\sum_{u\in N^-(x)} |\c{B}_{u}|\ge \sum_{u\in N^-(x)\sm \{u'\}} |P(u)|=|P(x)|-|P(u')|,\]
	with this last equality using Proposition~\ref{Pu Sets}(v).  This contradicts the assumption that $|\c{B}|<P(x)-\max_{u\in N^-(x)} |P(u)|$, so no singular vertex $u'$ can exist.  This same argument works if no such vertex $u'$ exists, proving the claim.
	
	Now let $u,u'\in N^-(x)$ be distinct vertices with $|\c{B}_u|<|P(u)|$ and $|\c{B}_{u'}|<|P(u')|$.  In particular, this size condition means there exists some $a\in P(u)$ and $a'\in P(u')$ such that $a\notin \{B(u):B\in \c{B}_{u}\}$ and $a'\notin \{B(u'):B\in \c{B}_{u'}\}$.  This implies $a\notin \{B(x):B\in \c{B}\}$ since by Lemma~\ref{bracket facts}(ii), having $B(x)=a$ is only possible if the unique in-neighbor $u\in N^-(x)$ with $a\in P(u)$ has $B(u)=a$, but by definition of $\c{B}_u$ and $a$, any $B\in \c{B}$ with $B(x)=B(u)$ has $B(u)\ne a$.  The same argument shows $a'\notin \{B(x):B\in \c{B}\}$ as well. Observe that $x=x_{a,a'}$ due to the existence of the in-neighbors $u,u'$, so this in total implies $\c{B}$ is not $\sig$-resolving by Proposition~\ref{necessary condition}, giving the desired contradiction.
\end{proof}

\section{Upper Bounds}\label{sec:upper}

\subsection{Additional Preliminaries}
Our upper bounds rely on ``lifting'' resolving sets from smaller tournaments $\tour'\sub \tour$, similar to the argument from our proof sketch in Section~\ref{sec:sketch}.  We state our main result in this direction in Proposition~\ref{lifting brackets}, for which we need the following definitions.

\begin{defn}
	Given a single-elimination tournament $\tour$, scoring system $\sig$, and vertex $u\in V(\tour)$, we define $\tour_u$ to be the digraph obtained from $\tour$ by deleting any vertex $v$ with $P(v)\not\sub P(u)$, and we define $\sig_u$ to be the restriction of $\sig$ to $M(\tour)\cap V(\tour_u)$.
\end{defn}
By Proposition~\ref{Pu Sets}(iii), $\tour_u$ is equivalently defined to be the digraph obtained by keeping all of the vertices of $\tour$ which have a directed walk to $u$.
\begin{lem}\label{touru are tournaments}
	If $\tour$ is a single-elimination tournament, $\sig$ a scoring system, and $u\in V(\tour)$, then  $\tour_u$ is a single-elimination tournament with $N^-_{\tour_u}(v)=N^-_{\tour}(v)$ for all $v\in V(\tour_u)$ and $\sig_u$ is a scoring system of $\tour_u$.
\end{lem}
\begin{proof}
	That $\tour_u$ has no directed cycles and has $|N^+_{\tour_u}(v)|\le 1$ for any $v$ follows from these properties holding for $\tour\supseteq \tour_u$.   We prove the the remaining properties through a series of claims.
	\begin{claim}
		The vertex $u$ is the unique sink of $\tour_u$.
	\end{claim}
	\begin{proof}
		Every $v\in V(\tour_u)$ has $P(v)\sub P(u)$ by definition, so there exists a directed walk from $v$ to $u$ by Proposition~\ref{Pu Sets}(iii) and hence $u$ is the only possible sink that $\tour_u$ could have.  On the other hand, if $u$ had some out-neighbor $v$ in $\tour_u$ then this would imply $P(u)\subsetneq P(v)$ by Proposition~\ref{Pu Sets}(iii) and (ii), a contradiction to our definition of $\tour_u$.
	\end{proof}
	\begin{claim}\label{Claim tour_u}
		We have $N^-_{\tour_u}(v)=N^-_{\tour}(v)$ for every $v\in V(\tour_u)$.
	\end{claim}
	\begin{proof}
		If $v\in V(\tour_u)$ and $u'\in N^-_{\tour}(v)$ then $P(u')\sub P(v)\sub P(u)$ by Proposition~\ref{Pu Sets}(v) and the definition of $\tour_u$, implying that $u'\in V(\tour_u)$ and hence $u'\in N^-_{\tour_u}(v)$.  This means $N^-_{\tour}(v)\sub N^-_{\tour_u}(v)$, and the other inclusion holds simply because $\tour_u$ is a sub-digraph of $\tour$. 
	\end{proof}
	Note that this last claim implies $|N^-_{\tour_u}(v)|=|N^-_{\tour}(v)|\ne 1$ for any $v$, which combined with the first claim and the observations at the start of the proof imply that $\tour_u$ is indeed a single-elimination tournament.  Moreover, Claim~\ref{Claim tour_u} exactly gives the second part of the lemma.
	
	Finally, we prove that $\sig_u$ is a scoring system for $\tour_u$, i.e.\ that its domain (which by definition is $M(\tour)\cap V(\tour_u)$) equals $M(\tour_u)$.  And indeed, Claim~\ref{Claim tour_u} implies that a vertex in $\tour_u$ is not a source if and only if it is not a source in $\tour$, proving $M(\tour)\cap V(\tour_u)=M(\tour_u)$.
\end{proof}
We need one more lemma before stating and proving our main result around $\tour_u$.

\begin{lem}\label{touru has at most one bad inneighbor}
	If $\tour$ is a single-elimination tournament and $u\in V(\tour)$, then every match $x\in M(\tour)\sm  V(\tour_u)$ has at most one in-neighbor $v\in N^-(x)$ with $P(v)\cap P(u)\ne \emptyset$.
\end{lem}
\begin{proof}
	Assume for contradiction that there existed two such in-neighbors $v,v'$.  Note that this means $P(u)\not \subseteq P(v)$ since $P(u)\cap P(v')\ne \emptyset$ and $P(v),P(v')$ are disjoint by Proposition~\ref{Pu Sets}(v).  It follows by Proposition~\ref{Pu Sets}(iv) that $P(v)\subsetneq P(u)$, which by Proposition~\ref{Pu Sets}(iii) means there is a directed walk $(u_1,\ldots,u_t)$ with $t\ge 2$ from $v$ to $u$.  Note that necessarily $u_2=x$ since $x$ is the unique out-neighbor of $v$, so $(u_2,\ldots,u_t)$ is a directed walk from $x$ to $u$, implying $P(x)\sub P(u)$ by Proposition~\ref{Pu Sets}(iii), a contradiction to the assumption $x\notin V(\tour_u)$.
\end{proof}

Our main result for these $\tour_u$ tournaments is a method which reduces finding resolving sets of $\tour$ to finding resolving sets of $\tour_u$.  Here and for the rest of the paper we adopt the convention that brackets $\widetilde{B}$ with tildes above them refer to brackets of $\tour_u$.
\begin{prop}\label{lifting brackets}
	Let $\tour$ be a single-elimination tournament, $\sig$ a scoring system, $u\in V(\tour)$, $\c{B}=\{B_1,\ldots,B_r\}$ a set of brackets, and $\widetilde{B}_i$ the restriction of each $B_i$ to $\tour_u$.  If:
	\begin{itemize}
		\item[(a)] $\{\widetilde{B}_1,\ldots,\widetilde{B}_r\}$ is a $\sig_u$-resolving set for $\tour_u$,
		\item[(b)] $B_i(x)\notin P(u)$ for all $x\in M(\tour)\sm  V(\tour_u)$, and 
		\item[(c)] For every pair of brackets $B,B'$ with $\score_\sig(B_i,B)=\score_\sig(B_i,B')$ for all $i$, we have for all $a\in P(\tour)\sm P(u)$ and every $x\in M(\tour)$ that $B(x)=a$ if and only if $B'(x)=a$;
	\end{itemize}
	then $\c{B}$ is a $\sig$-resolving set.
\end{prop}

\begin{proof}
	To aid with this proof, we say a pair of brackets $B,B'$ are $\sig$-indistinguishable if we have $\score_\sig(B_i,B)=\score_\sig(B_i,B')$ for all $i$, and for every bracket $B$ of $\tour$ we let $\widetilde{B}$ denote its restriction to $\tour_u$.  Note that $\widetilde{B}$ is always a bracket for $\tour_u$ if $B$ is for $\tour$ since $N^-_{\tour_u}(v)=N^-_{\tour}(v)$ for $v\in V(\tour_u)$.
	
	We aim to show that any two $\sig$-indistinguishable brackets equal each other on every match, which we do via a series of claims.
	\begin{claim}
		If $B,B'$ are $\sig$-indistinguishable, then for every $x\in M(\tour)\sm V(\tour_u)$ and $i$ we have $B(x)=B_i(x)$ if and only if $B'(x)=B_i(x)$.
	\end{claim}
	\begin{proof}
		By (b) any such $x$ has $a:=B_i(x)\notin P(u)$, so by (c) having $B(x)=B_i(x)=a$ is equivalent to having $B'(x)=B_i(x)=a$ as desired.
	\end{proof}
	\begin{claim}
		If $B,B'$ are $\sig$-indistinguishable, then for all $1\le i\le r$ there exists a real number $\al_i$ such that
		\[\score_\sig(B_i,B)=\al_i+\score_{\sig_u}(\widetilde{B}_i,\widetilde{B}),\]
		and
		\[\score_\sig(B_i,B')=\al_i+\score_{\sig_u}(\widetilde{B}_i,\widetilde{B}').\]
	\end{claim}
	\begin{proof}
		Define \[\al_i=\sum_{x\in M(\tour)\sm V(\tour_u):B_i(x)=B(x)} \sig(x),\]
		noting that by definition
		\[\score_\sig(B_i,B)=\sum_{x\in M(\tour)\sm V(\tour_u):B_i(x)=B(x)} \sig(x)+\sum_{x\in M(\tour)\cap V(\tour_u):B_i(x)=B(x)} \sig(x)=\al_i+\score_{\sig_u}(\widetilde{B}_i,\widetilde{B}).\]
		On the other hand, we also have \[\al_i=\sum_{x\in M(\tour)\sm V(\tour_u):B_i(x)=B'(x)} \sig(x),\] since the previous claim implies that $B_i(x)=B(x)$ if and only if $B_i(x)=B'(x)$ for $x\notin V(\tour_u)$, so by the same reasoning as above we find
		\[\score_\sig(B_i,B')=\al_i+\score_{\sig_u}(\widetilde{B}_i,\widetilde{B}'),\]
		proving the claim.
	\end{proof}
	\begin{claim}\label{Claim lifting and restrictions}
		If $B,B'$ are $\sig$-indistinguishable, then $\widetilde{B}=\widetilde{B}'$.
	\end{claim}
	\begin{proof}
		The previous claim implies that $\score_{\sig_u}(\widetilde{B}_i,\widetilde{B})=\score_{\sig_u}(\widetilde{B}_i,\widetilde{B}')$ for all $i$, which means $\widetilde{B}=\widetilde{B}'$ since $\{\widetilde{B}_1,\ldots,\widetilde{B}_r\}$ is $\sig_u$-distinguishing by (a).
	\end{proof}
	Now assume for contradiction that there exist $\sig$-indistinguishable brackets $B,B'$ with $B\ne B'$.  Because $B(a)=a=B'(a)$ for all $a\in P(\tour)$ by definition of brackets, we must have $B(x)\ne B'(x)$ for some match $x$, and we choose such a match with $|P(x)|$ as small as possible.  Note that $B(x)=B'(x)$ for $x\in V(\tour_u)$ by Claim~\ref{Claim lifting and restrictions}, so we must have $x\notin V(\tour_u)$.  Similarly by (c) we must have $B(x),B'(x)\in P(u)$.  Because $B(x)=B(v)\in P(v)$ for some $v\in N^-(x)$ by definition of brackets and Lemma~\ref{bracket facts}(i), the fact that $B(x)\in P(u)$ implies from Lemma~\ref{touru has at most one bad inneighbor} that $B(x)=B(v')$ for the unique in-neighbor $v'$ of $x$ with $P(v')\cap P(u)\ne \emptyset$.  The same line of reasoning implies that $B'(x)=B'(v')$.  But this now implies that
	\[B(x)=B(v')=B'(v')=B'(x),\]
	where the middle equality is immediate if $v'\in P(\tour)$ and otherwise follows by the minimality of $|P(x)|$ if $v'$ is a match.  This gives a contradiction to our choice of $x$, so we conclude that every $\sig$-indistinguishable pair has $B=B'$, proving that $\c{B}$ is a $\sig$-resolving set as desired. 
\end{proof}

We close with a fact which will be useful for constructing brackets as in Proposition~\ref{lifting brackets}, where here we recall that $\widehat{B}_a$ denotes the bracket obtained from $\widehat{B}$ by setting $\widehat{B}_a(v)=a$ for all $v$ with $a\in P(v)$.
\begin{lem}\label{constructing lifting brackets}
	Let $\tour$ be a single-elimination tournament, $u\in V(\tour)$, and $\widetilde{B}$ a bracket for $\tour_u$.  Then there exists a bracket $\widehat{B}$ for $\tour$ such that the restriction of $\widehat{B}$ to $\tour_u$ equals $\widetilde{B}$ and such that $\widehat{B}(x)\notin P(u)$ for all $x\in M(\tour)\sm V(\tour_u)$.  Moreover, the bracket $\widehat{B}_a$ continues to satisfy these conditions for every $a\notin P(u)$
\end{lem}
\begin{proof}
	We define $\widehat{B}:V(\tour)\to P(\tour)$ iteratively as follows.  We begin by assigning $\widehat{B}(v)=\widetilde{B}(v)$ for all $v\in V(\tour_u)$ as well as $\widehat{B}(a)=a$ for all $a\in P(\tour)$.  Iteratively if $x\in M(\tour)$ is a match which is unassigned with $|P(x)|$ as small as possible,  we define $\widehat{B}(x)=\widehat{B}(v)$ for any $v\in N^-(x)$ with $P(v)\cap P(u)=\emptyset$; note that such a $v$ must exist by Lemma~\ref{touru has at most one bad inneighbor} and that $\widehat{B}(v)$ is well-defined by our choice of unassigned $x$ with $|P(x)|$ as small as possible.  It is straightforward to check that the resulting $\widehat{B}$ is a bracket with the desired properties, proving the first half of the result.
	
	Now consider $\widehat{B}_a$ for some $a\notin P(u)$.  The only way this could fail the conditions of the lemma is if $\widehat{B}_a(v)\ne \widetilde{B}(v)$ for some $v\in V(\tour_u)$.  Because $\widehat{B}(v)=\widetilde{B}(v)$ for $v\in V(\tour_u)$, we must have $a\in P(v)$ by definition of $\widehat{B}_a$.  But each $v\in V(\tour_u)$ has $P(v)\sub P(u)$ by definition of $\tour_u$, a contradiction to $a\in P(v)$ and $a\notin P(u)$, proving the result.
\end{proof}

We now prove our remaining results which we split into three subsections for convenience.

\subsection{Tightness of Lower Bounds}
We begin by proving that there exist scoring systems such that our general lower bound for $\dim(\tour,\sig)$ is tight.  To this end, we say that a scoring system $\sig$ has \textit{distinct subset sums} if having $\sum_{x\in M} \sig(x)=\sum_{x\in M'} \sig(x)$ for two sets of matches $M,M'$ implies that $M=M'$.   We record the following immediate consequence of this definition for ease of use.
\begin{lem}\label{distinct sums fact}
	If $\tour$ is a single-elimination tournament, if $\sig$ is a scoring system with distinct subset sums, and if $B_0,B,B'$ are brackets with $\score_\sig(B_0,B)=\score_\sig(B_0,B')$, then for every match $x$ we have $B_0(x)=B(x)$ if and only if $B_0(x)=B'(x)$. 
\end{lem}
\begin{proof}
	Let $M$ be the set of matches with $B_0(x)=B(x)$ and similarly define $M'$.  Then by definition
	\[\sum_{x\in M} \sig(x)=\score_\sig(B_0,B)=\score_\sig(B_0,B')=\sum_{x\in M'} \sig(x),\]
	and by hypothesis on $\sig$ this implies $M=M'$ as desired.
\end{proof}

We now determine $\dim(\tour,\sig)$ exactly for scoring system with distinct subset sums.

\begin{lem}\label{tight deterministic}
	If $\tour$ is a single-elimination tournament with at least 2 players and if $\sig$ is a scoring system with distinct subset sums, then \[\dim(\tour,\sig)= \max_{x\in M(\tour)}\left( |P(x)|-\max_{u\in N^-(x)} |P(u)|\right).\]
\end{lem}
\begin{proof}
	The lower bound follows from Proposition~\ref{lower bound deterministic}.  We prove the upper bound by induction on the number of vertices of $\tour$.  Since we consider $\tour$ with at least 2 players, the base case $|V(\tour)|=3$ has a unique single-elimination tournament which can easily be checked to have $\dim(\tour,\sig)=1$ which is the bound we wish to show in this case.
	
	Assume we have proven the result up to some number of vertices $\nu\ge 4$.  Let $\tour$ be a $\nu$-vertex single-elimination tournament, and for notational convenience let
	\[r:= \max_{x\in M(\tour)}\left( |P(x)|-\max_{u\in N^-(x)} |P(u)|\right).\]  Let $z$ denote the sink of $\tour$ and $w\in N^-(z)$ an arbitrary in-neighbor.  We aim to apply Proposition~\ref{lifting brackets} to $\tour_w$.  For this we observe that $\tour_w$ has strictly fewer vertices than $\tour$ since $z\notin V(\tour_w)$ and that $\sig_w$ has distinct subset sums since $\sig$ does.
	
	\begin{claim}
		There exists a $\sig_w$-resolving set of brackets $\{\widetilde{B}^1,\ldots,\widetilde{B}^r\}$ for $\tour_w$. 
	\end{claim}
	\begin{proof}
		If $\tour_w$ has only 1 player then one can trivially take $\widetilde{B}^i$ to be the unique bracket on $\tour'$.  Otherwise, because $\tour_w$ has fewer vertices than $\tour$ we inductively know that we can find a $\sig_w$-resolving set $\widetilde{\c{B}}$ of size at most  \[\max_{v\in M(\tour_w)}\left( |P(v)|-\max_{u\in N^-_{\tour_w}(v)} |P(u)|\right)=\max_{v\in M(\tour)\cap V(\tour_w)}\left( |P(v)|-\max_{u\in N^-_{\tour}(v)} |P(u)|\right)\le r,\]
		where our equality implicitly used the first two parts of Lemma~\ref{touru are tournaments}.   As such, we can write the elements of this small $\sig_w$-resolving set as $\widetilde{\c{B}}=\{\widetilde{B}^1,\ldots,\widetilde{B}^r\}$ after possibly repeating some of the elements in this list.
	\end{proof}
	In conjunction with this, observe that $|P(\tour)\sm P(w)|=|P(z)|-|P(w)|\le r$ by definition of $r$.  As such, we can write $P(\tour)\sm P(w)=\{a_1,\ldots,a_r\}$  after possibly repeating some of the elements in this list.

	By Lemma~\ref{constructing lifting brackets}, for each $1\le i\le r$ there exist brackets $\widehat{B}^i$ for $\tour$ which equal $\widetilde{B}^i$ when restricted to $V(\tour_u)$ and which have $\widehat{B}^i(x)\notin P(u)$ for all $x\in M(\tour)\sm V(\tour_u)$, and moreover these properties continue to hold for the bracket $\widehat{B}_{a_i}^i$.  As such, the brackets  $\c{B}=\{\widehat{B}_{a_1}^1,\ldots,\widehat{B}_{a_r}^r\}$ satisfy (a) and (b) of Proposition~\ref{lifting brackets}, so we need only check that they satisfy (c) as well to prove that this is a $\sig$-resolving set of the desired size.
	
	To this end, assume for contradiction that there exist $B,B'$ with $\score_\sig(\widehat{B}^i_{a_i},B)=\score_\sig(\widehat{B}_{a_i}^i,B')$ for all $i$ and that there exists some $a_i\in P(\tour)\sm P(u)$ and $x\in M(\tour)$ with, say, $B(x)=a_i$ and $B'(x)\ne a_i$.  Now having $B(x)=a_i$ implies that $a_i\in P(x)$ by Lemma~\ref{bracket facts}(i), so by definition $\widehat{B}_{a_i}^i(x)=a_i=B(x)$ but $\widehat{B}_{a_i}^i(x)\ne B'(x)$, a contradiction to Lemma~\ref{distinct sums fact}.  We conclude the result.
\end{proof}

We conclude this subsection by formally completing the proof of Theorem~\ref{lower bound}.

\begin{proof}[Proof of Theorem~\ref{lower bound}]
	The bound follows from Proposition~\ref{lower bound deterministic} and the tightness from Lemma~\ref{tight deterministic}.
\end{proof}

\subsection{Metric Dimension Upper Bounds}
We begin with a technical auxiliary lemma which we will need momentarily.
\begin{lem}\label{a implies b}
	Let $\tour$ be a single-elimination tournament on at least 2 vertices, $a\in P(\tour)$, and $x_a$ the unique out-neighbor of $a$.  If $b\in P(x_a)\sm \{a\}$, then $b\in P(x)$ for every match $x$ with $a\in P(x)$.
\end{lem}
\begin{proof}
	If $b\in P(x_a)$ then by definition this means there exists a directed walk $(u'_1,\ldots,u'_s)$ from $b$ to $x_a$.  Similarly for every match $x$ with $a\in P(x)$ there exists a directed walk $(u_1,\ldots,u_t)$ from $a$ to $x$, where necessarily $t\ge 2$ since $a$ is not a match and also $u_2=x_a$ since this is the unique out-neighbor of $a$.  But this means that $(u'_1,\ldots,u'_{s-1},u'_s=u_2,u_3,\ldots,u_t)$ is a directed walk from $b$ to $x$, proving $b\in P(x)$ as desired.
\end{proof}

To apply Proposition~\ref{lifting brackets}, we must verify its most difficult hypothesis of (c).  We can achieve this for a given player $a$ provided two brackets satisfying certain properties lie in $\c{B}$.

\begin{lem}\label{favorite team}
	Let $\tour$ be a single-elimination tournament with at least 2 vertices and $\widehat{B}$ a bracket.  If $a,b$ are distinct players such that the out-neighbor $x_a\in N^+(a)$ has $b\in \{\widehat{B}(v):v\in N^-(x_a)\}$, then for any brackets $B,B'$ and scoring system $\sig$ satisfying that $\score_\sig(\widehat{B}_a,B)=\score_\sig(\widehat{B}_a,B')$ and $\score_\sig(\widehat{B}_b,B)=\score_\sig(\widehat{B}_b,B')$, we have for all $x\in M(\tour)$ that $B(x)=a$ if and only if $B'(x)=a$ .
\end{lem}
\begin{proof}
	We begin with a few observations about our setup.
	\begin{claim}\label{Claim favorite team nice vertex}
		There exists an in-neighbor $v\in N^-(x_a)$ with $\widehat{B}(v)=b$ and with $b\in P(v)$ and $a\notin P(v)$.
	\end{claim}
	\begin{proof}
		The existence of $v\in N^-(x_a)$ with $\widehat{B}(v)=b$ is by hypothesis, and this immediately implies $b\in P(v)$ by Lemma~\ref{bracket facts}(i).  Moreover, because $a\in N^-(x_a)$ (with necessarily $a\ne v$ since $b\notin P(a)$), we must have $a\notin P(v)$ by Proposition~\ref{Pu Sets}(v).
	\end{proof}
	From now on any reference to $v$ refers to the vertex guaranteed by Claim~\ref{Claim favorite team nice vertex}.
	
	\begin{claim}\label{Claim favorite team a implies b}
		Every match $x$ with $a\in P(x)$ has $b\in P(x)$.
	\end{claim}
	\begin{proof}
		By Claim~\ref{Claim A size} and Proposition~\ref{Pu Sets}(v) we have $b\in P(v)\sub P(x_a)\sm \{a\}$, so the claim follows from Lemma~\ref{a implies b}.

	\end{proof}
	
	\begin{claim}
		We have $\widehat{B}_a(x)=\widehat{B}_b(x)$ for every match $x$ with $a\notin P(x)$.
	\end{claim}
	\begin{proof}
		If $x$ is a match with $a,b\notin P(x)$ then $\widehat{B}_a(x)=\widehat{B}(x)=\widehat{B}_b(x)$ by definition, so it suffices to prove this for matches with $a\notin P(x)$ and $b\in P(x)$.  Because $\widehat{B}(v)=b$, we must have $\widehat{B}(x)=b$ for all $x$ with $P(x)\sub P(v)$ by Lemma~\ref{bracket facts}(ii), which means $\widehat{B}_a(x)=\widehat{B}(x)=b=\widehat{B}_b(x)$ for all such $x$.  The only remaining case to check then is when $P(x)\not \sub P(v)$.  Because $b\in P(x)\cap P(v)$, any such $x$ has $P(v)\subsetneq P(x)$ by Proposition~\ref{Pu Sets}(iv), so by Proposition~\ref{Pu Sets}(iii) there is a directed walk $(u_1,\ldots,u_t)$ from $v$ to $x$.  Note that $u_2=x_a$ since this is the unique out-neighbor of $v$, so $(u_2,\ldots,u_t)$ is a directed walk from $x_a$ to $x$, implying by Proposition~\ref{Pu Sets}(iii) that $a\in P(x_a)\sub P(x)$, a contradiction to the assumption $a\notin P(x)$.  We conclude the result.
	\end{proof}
	Let $M_a\sub M(\tour)$ denote the set of matches $x$ with $a\in P(x)$.  From this claim and the definition of $\score_\sig$, we see that for any bracket $B^*$ we have
		\begin{equation}\score_\sig(\widehat{B}_a,B^*)-\score_\sig(\widehat{B}_b,B^*)=\sum_{x\in M_a:B^*(x)=a}\sig(x)-\sum_{x\in M_a:B^*(x)=b}\sig(x).\label{eq:difference in scores}\end{equation}

	Now assume for contradiction that the lemma is false for some match $y$ and choose this $y$ with $|P(y)|$ as small as possible.  Without loss of generality we may assume that $B(y)=a$ and $B'(y)\ne a$.
	
	We claim that $B(x)\ne b$ for every $x\in M_a$.  Indeed if $B(x)=b$ for some $x\in M_a$ then this means $a,b\in P(x)$ by Lemma~\ref{bracket facts}(i) and the definition of $M_a$.  Similarly we have $a,b\in P(y)$ due to Lemma~\ref{bracket facts}(i) and Claim~\ref{Claim favorite team a implies b}.  By Proposition~\ref{Pu Sets}(iv) we have that either $P(x)\sub P(y)$ or $P(y)\sub P(x)$.  In the first case because $a\in P(x)\sub P(y)$ and $B(y)=a$ we must have $B(x)=a$ by Lemma~\ref{bracket facts}(ii), a contradiction.  The same argument holds if $P(y)\sub P(x)$, proving the claim.  With this, we have by \eqref{eq:difference in scores} that
	\[\score_\sig(\widehat{B}_a,B)-\score_\sig(\widehat{B}_b,B)= \sum_{x\in M_a:B(x)=a}\sig(x)\ge \sum_{x\in M_a: P(x)\sub P(y)} \sig(x),\]
	with this last step using Lemma~\ref{bracket facts}(ii) to ensure $B(x)=a$ for all $x\in M_a$ with $P(x)\sub P(y)$.
	
	We claim that $B'(x)\ne a$ for any $x\in M_a$ with $P(x)=P(y)$ or $P(x)\not\sub P(y)$.  Indeed, the first part follows since Proposition~\ref{Pu Sets}(ii) implies $P(x)=P(y)$ only can happen for $x=y$ and we assumed $B'(y)\ne a$.  For the second part, having $x\in M_a$ means $a\in P(x)\cap P(y)$, so $P(x)\not \sub P(y)$ would imply $P(y)\sub P(x)$ by Proposition~\ref{Pu Sets}(iv).  But now $B'(x)=a$ would imply $B'(y)=a$ by Lemma~\ref{bracket facts}(ii), a contradiction.  This proves the claim, and hence by \eqref{eq:difference in scores} we have 
	\[\score_\sig(\widehat{B}_a,B')-\score_\sig(\widehat{B}_b,B')\le \sum_{x\in M_a:B'(x)=a}\sig(x)= \sum_{x\in M_a: P(x)\subsetneq P(y)}\sig(x)< \sum_{x\in M_a: P(x)\sub P(y)} \sig(x),\]
	with this last step using that $\sig(y)>0$ by definition of scoring systems.
	
	In total we find that 
	\[\score_\sig(\widehat{B}_a,B)-\score_\sig(\widehat{B}_b,B)>\score_\sig(\widehat{B}_a,B')-\score_\sig(\widehat{B}_b,B'),\]
	which implies that either  $\score_\sig(\widehat{B}_a,B)\ne \score_\sig(\widehat{B}_a,B')$ or $\score_\sig(\widehat{B}_b,B)\ne \score_\sig(\widehat{B}_b,B')$, a contradiction to our hypothesis.  We conclude the result.
\end{proof}

To effectively use this lemma we will need to strengthen Lemma~\ref{constructing lifting brackets} as follows.

\begin{lem}\label{constructing better lifting brackets}
	Let $\tour$ be a single-elimination tournament on at least 2 vertices and $u\in V(\tour)$.  If $A\sub P(\tour)\sm P(u)$ is a set of players with the property that for every $a\in A$ there exists $b\in A\sm \{a\}$ such that $b\in P(x)$ for every match $x$ with $a\in P(x)$, then for any bracket $\widetilde{B}$ of $\tour_u$ there exists a bracket $\widehat{B}$ of $\tour$ such that the following holds:
	\begin{itemize}
		\item[(i)] $\widehat{B}$ equals $\widetilde{B}$ when restricted to $V(\tour_u)$ and $\widehat{B}(x)\notin P(u)$ for all $x\in M(\tour)\sm P(\tour_u)$,
		\item[(ii)] Condition (i) continues to hold for $\widehat{B}_c$ for all $c\notin P(u)$, and
		\item[(iii)] For every $a\in A$ there exists some $b\in A\sm \{a\}$ such that the unique out-neighbor $x_a\in N^+(a)$ has $b\in \{\widehat{B}(v):v\in N^-(x_a)\}$.
	\end{itemize}
\end{lem}
\begin{proof}
	By Lemma~\ref{constructing lifting brackets} there exists a bracket $\widehat{B}^0$ for $\tour$ which satisfies (i), and moreover Lemma~\ref{constructing lifting brackets} shows that any bracket satisfying (i) also satisfies (ii).  Letting $A=\{a_1,\ldots,a_t\}$, we iteratively define $\widehat{B}^i=\widehat{B}_{a_i}^{i-1}$ and take $\widehat{B}:=\widehat{B}^t$.  That is, we define $\widehat{B}$ by taking $\widehat{B}^0$ and changing it so that $a_1$ wins every match it can, then changing this so that $a_2$ wins every match that it can, and so on.  Note that Lemma~\ref{constructing lifting brackets} guarantees that (i) and (ii) iteratively holds for all $\widehat{B}^i$, so it remains to check (iii).
	
	To this end, consider any $a\in A$.  Because $a\in P(x_{a})$, we have by hypothesis that there exists some $b'\in P(x_{a})\sm \{a\}$.  By Proposition~\ref{Pu Sets}(v) this means there exists some $v\in N^-(x_a)$ with $b'\in P(v)$ and necessarily $a\notin P(v)$ since $a\in N^-(x_a)$ has $b'\notin P(a)=\{a\}$.  Observe that by construction, $\widehat{B}(v)=a_j$ for the largest $j$ with $a_j\in P(v)$, noting that such a $j$ must exist because $b'\in A\cap P(v)$ and that $a_j\ne a$ since $a\notin P(v)$.  We conclude that (iii) holds with $b=a_j$.
\end{proof}

We now state our main technical result for this subsection.

\begin{prop}\label{metric upper bound given partition}
	Let $\tour$ be a single-elimination tournament, $u\in V(\tour)$ a vertex, $\c{B}_u$ a non-empty set of brackets for $\tour_u$, and $\c{A}$ a partition of $P(\tour)\sm P(u)$ such that for every $a\in P(\tour)\sm P(u)$, the set $A\in \c{A}$ with $a\in A$ has some $b\in A\sm \{a\}$ such that $b\in P(x)$ for every match $x$ with $a\in P(x)$.  Then there exists a set of brackets $\c{B}$ for $\tour$ with
	\[|\c{B}|= |P(\tour)\sm P(u)|+\max\{0,|\c{B}_u|-|\c{A}|\},\]
	such that $\c{B}$ is a $\sig$-resolving set whenever $\c{B}_u$ is a $\sig_u$-resolving set.
\end{prop}
\begin{proof}
	The result is trivial if $\tour$ has a single vertex, so we assume from now on that $\tour$ has at least 2 vertices.  Let $\c{A}=\{A_1,\ldots,A_t\}$ and $\c{B}_u=\{\widetilde{B}^1,\ldots,\widetilde{B}^m\}$ be as in the hypothesis.  For some slight ease of notation we first prove the result under the assumption $t\le m$, after which we briefly discuss the minor changes needed to make this argument work for $t>m$.
	
	For each $1\le i \le t$, let $\widehat{B}^i$ denote the bracket obtained by applying Lemma~\ref{constructing better lifting brackets} with respect to $\widetilde{B}=\widetilde{B}^{i}$ and $A=A_i$, and for each $t<i\le m$ let $\widehat{B}^i$ denote the bracket obtained by applying Lemma~\ref{constructing better lifting brackets} to $\widetilde{B}=\widetilde{B}^i$ and $A=\emptyset$.  With this defined, we let
	\[\c{B}=\bigcup_{i=1}^t\{\widehat{B}_a^{i}:a\in A_i\}\cup \{\widehat{B}^{t+1},\widehat{B}^{t+2},\ldots,\widehat{B}^m\},\]
	where to be clear this latter set is empty if $t=m$.
	
	Observe that this set of brackets has size $|P(\tour)\sm P(u)|+m-t$ since $\c{A}$ forms a partition of $P(\tour)\sm P(u)$, and this is exactly the size we aimed to prove in the case $t\le m$.  It thus remains to show that $\c{B}$ is $\sig$-resolving whenever $\c{B}_u$ is $\sig_u$-resolving, for which we aim to apply Proposition~\ref{lifting brackets}.
	
	Condition (b) of Proposition~\ref{lifting brackets} that each $B\in \c{B}$ has $B(x)\notin P(u)$ for $x\in M(\tour)\sm V(\tour_u)$ follows from the fact that each $\widehat{B}^i$ bracket defined above comes from Lemma~\ref{constructing better lifting brackets} which satisfies conditions (i) and (ii) of that lemma.  Condition (a) of Proposition~\ref{lifting brackets} that restricting each bracket of $\c{B}$ to $\tour_u$ gives a $\sig_u$-resolving set for $\tour_u$ will follow if we can show that for each $\widetilde{B}^i\in \c{B}_u$  there exists $B\in \c{B}$ whose restriction equals $\widetilde{B}^i$.  And indeed, this holds with $B=\widehat{B}^i$ if $i>t$, and if $i\le t$ we let $a\in A_i$ be arbitrary and take $B=\widehat{B}^{i}_a\in \c{B}$.  This verifies condition (a).
	
	It remains then to show condition (c) of Proposition~\ref{lifting brackets}.  To this end, consider any $a\in P(\tour)\sm P(u)$, say with $a\in A_i$.  By Lemma~\ref{constructing better lifting brackets}, there exists some $b\in A_i\sm \{a\}$ such that the unique out-neighbor $x_a\in N^+(a)$ has $b\in \{\widehat{B}^i(v):v\in N^-(x_a)\}$.  It follows from Lemma~\ref{favorite team} that any two brackets $B,B'$ with the same $\sig$-scores with respect to $\widehat{B}^i_a,\widehat{B}^i_b\in \c{B}$ have $B(x)=a$ if and only if $B'(x)=a$.  We conclude that (c) holds, and hence that $\c{B}$ is $\sig$-distinguishing by Proposition~\ref{lifting brackets}.
	
	We now briefly consider the case  $m\le t$.  For each $1\le i \le m$, let $\widehat{B}^i$ denote the bracket obtained by applying Lemma~\ref{constructing better lifting brackets} with respect to $\widetilde{B}=\widetilde{B}^{i}$ and $A=A_i$, and for each $m<i\le t$ let $\widehat{B}^i$ denote the bracket obtained by applying Lemma~\ref{constructing better lifting brackets} with $\widetilde{B}=\widetilde{B}^m$.  Let
	\[\c{B}=\bigcup_{i=1}^t\{\widehat{B}_a^{i}:a\in A_i\}.\]
	Note that this set has size $|P(\tour)\sm P(u)|$, and essentially the same analysis as before shows that the conditions of Proposition~\ref{lifting brackets} are satisfied, proving the result.	
\end{proof}
We note that it is possible to refine the bound of Proposition~\ref{metric upper bound given partition} under certain circumstances when $u$ is not an in-neighbor of the sink of $\tour$, but we refrain from stating such a result due to its complexity. To use Proposition~\ref{metric upper bound given partition}, we need to construct large partitions $\c{A}$ as in its statement, and it turns out to be relatively easy to construct such partitions of the best possible size.
\begin{lem}\label{partitions exist}
	Let $\tour$ be a single-elimination tournament with at least 2 vertices and $u\in V(\tour)$.  If there does not exist $x\in N^+(u)$ with $N^-(x)=\{a,u\}$ for some $a\in P(\tour)$, then there exists a partition $\c{A}$ of $P(\tour)\sm P(u)$ such that: 
	\begin{itemize}
		\item[(i)] For all $a\in A\in \c{A}$, there exists $b\in A\sm \{a\}$ such that $b\in P(x)$ for every match $x$ with $a\in P(x)$, and
		\item[(ii)] We have \[|\c{A}|=\sum_{x\in M(\tour)\sm V(\tour_u)} \left\lfloor \frac{|N^-(x)\cap P(\tour)\sm P(u)|}{2} \right\rfloor.\]
	\end{itemize}
\end{lem}
\begin{proof}
	Let $\c{A}_1$ be the partition of $P(\tour)\sm P(u)$ into singletons. Iteratively if there exist $\{a\},\{b\}\in \c{A}_1$ with $N^+(a)=N^+(b)$ then we remove $\{a\},\{b\}$ from $\c{A}_1$ and replace them with $\{a,b\}$, and we let $\c{A}_2$ be the resulting set system after this process terminates. 
	
	\begin{claim}\label{Claim A size}
		The number of 2-element subsets of $\c{A}_2$ is $\sum_{x\in M(\tour)\sm V(\tour_u)} \lfloor \frac{|N^-(x)\cap P(\tour)\sm P(u)|}{2} \rfloor$.
	\end{claim}
	\begin{proof}
		Every $\{a,b\}\in \c{A}_2$ has $\{a,b\}\sub N^-(x)$ for some $x\in M(\tour)\sm V(\tour_u)$ by construction (since having $x\in V(\tour_u)$ would imply the same for $a,b$), and the number of such pairs coming from a given $x$ is exactly $ \lfloor \frac{|N^-(x)\cap P(\tour)\sm P(u)|}{2} \rfloor$, proving the claim.
	\end{proof}
	We now want to move each singleton of $\c{A}_2$ into an appropriate pair in $\c{A}_2$, for which we use the following.
	\begin{claim}
		Letting $x_a$ denote the unique vertex in $N^+(a)$, we have for every $a\in P(\tour)\sm P(u)$ that there exists $x\in M(\tour)\sm V(\tour_u)$ with $P(x)\sub P(x_a)$ such that $|N^-(x)\cap P(\tour)\sm P(u)|\ge 2$.
	\end{claim}
	\begin{proof}
		Let $x$ be a match with $|P(x)|$ as small as possible amongst matches with $P(x)\sub P(x_a)$ and $P(x)\not \sub P(u)$, noting that such a match exists since $x_a$ satisfies this condition due to the assumption $a\notin P(u)$.
		
		Observe that if there exists $v\in N^-(x)$ with $P(v)\sub P(u)$ then we must have $v=u$ since $P(v)\sub P(u)$ implies by Proposition~\ref{Pu Sets}(iii) that there is a directed walk from $v$ to $u$, and if $v\ne u$ then this implies there is a directed walk from $x$ to $u$ and hence $P(x)\sub P(u)$, a contradiction.  We also have that $N^-(x)\sm \{u\}\sub P(\tour)$, for if there were a match $y\in N^-(x)\sm \{u\}$ then this would satisfy $P(y)\not \sub P(u)$ by the first half of this observation and that $P(y)\subsetneq P(x)\sub P(x_a)$ by Proposition~\ref{Pu Sets}(v), a contradiction to our choice of $x$.
		
		From these observations, we see that the result holds if $|N^-(x)\sm \{u\}|\ge 2$.  Because  every match in a single-elimination tournament has at least 2 in-neighbors, we see that this can only fail to hold if $N^-(x)=\{a',u\}$ for some player $a'$, but we assumed no such match existed in the hypothesis of the lemma, proving the claim.
	\end{proof}
	
	Now for each $\{a\}\in \c{A}_2$ we let $x$ be a match as in the previous claim and observe that by construction there must exist some $\{a',b'\}\in \c{A}_2$ with $\{a',b'\}\sub N^-(x)$.  As such, for each $\{a\}\in \c{A}_2$ we can remove this singleton from the partition and then insert $a$ into one of these pairs $\{a',b'\}$ (possibly with multiple singletons being iteratively added to the same original pair in $\c{A}_2$), and we let $\c{A}$ denote the final set system obtained from this procedure.
	
	It is not difficult to see that $\c{A}$ is a partition of $P(\tour)\sm P(u)$ since we iteratively maintained this property through each alteration, and also that $|\c{A}|$ equals the number of pairs in $\c{A}_2$ which is the desired number we are looking for by Claim~\ref{Claim A size}.  Finally, (i) holds for any given $a$ by applying Lemma~\ref{a implies b} with $b$ either the unique player that $a$ was paired with in $\c{A}_2$ if such a vertex exists, or by taking $b$ to be any vertex in the pair $\{a',b'\}$ that $a$ was merged with as a singleton.
\end{proof}
Combining all of these results gives are main upper bound on metric dimensions of single-elimination tournaments.

\begin{thm}\label{upper bound technical}
	Let $\tour$ be a single-elimination tournament with $n$ players and $U\sub V(\tour)$ the set of vertices $u$ such that there exists no $x\in N^+(u)$ with $N^-(x)=\{a,u\}$ for some player $a\in P(\tour)$.  Then for every scoring system $\sig$ we have
	\[\dim(\tour,\sig)\le \min_{u\in U} n-|P(u)|+\max\{0,\dim(\tour_u,\sig_u)- \sum_{x\in M(\tour)\sm M(\tour_u)} \left\lfloor \frac{|N^-(x)\cap P(\tour)\sm P(u)|}{2} \right\rfloor\}.\]
\end{thm}
Actually, this same bound holds if we take the minimum over all $u\in V(\tour)$ since one can show that this minimum is always achieved by some $u\in U$, but we refrain from showing the formal details of this.
\begin{proof}
	The result is trivial if $n=1$, so assume $n\ge 2$.  Let $u\in U$ be a vertex which achieves the minimum, let $\c{B}_u$ be a $\sig_u$-resolving set of size $\dim(\tour_u,\sig_u)$, and let $\c{A}$ be the partition of $P(\tour)\sm P(u)$ guaranteed by Lemma~\ref{partitions exist} since $u\in U$.  If $\c{B}_u\ne \emptyset$ then Proposition~\ref{metric upper bound given partition} guarantees the existence of a $\sig$-resolving set of the desired size.  If $\c{B}_u=\emptyset$, then $\c{B}_u$ being a $\sig_u$-resolving set for $\tour_u$ implies that $\tour_u$ must have only a single vertex.  In this case we let $\c{B}_u'=\{B\}$ where $B$ is the unique bracket of $\tour_u$ and apply Proposition~\ref{metric upper bound given partition} with respect to $\c{B}_u'$ and $\c{A}$ to obtain a $\sig$-resolving set of size 
	\[n-|P(u)|+\max\{0,1-|\c{A}|\}=n-|P(u)|,\]
	with this equality using that $\c{A}$ is a partition of $P(\tour)\sm P(u)$ which is a non-empty set since $n\ge 2$ and $P(u)$ has only a single vertex.  This is exactly the bound we wished to establish in this case, proving the result.
\end{proof}

A simple application of Theorem~\ref{upper bound technical} gives our claimed bound of $\dim(\tour,\sig)\le n-1$.
\begin{proof}[Proof of Theorem~\ref{upper bound}]
	Let $\tour$ be a single-elimination tournament with $n$ players and $\sig$ an arbitrary scoring system.  If $n=1$ then $\dim(\tour,\sig)=0=n-1$, so from now on we may assume $n\ge 2$.   For this we use a weakened version of Theorem~\ref{upper bound technical}, namely that
	\begin{equation}\dim(\tour,\sig)\le \min_{u\in U} n-|P(u)|+\dim(\tour_u,\sig_u).\label{eq:weak upper bound}\end{equation}

	If $\tour$ contains a player $a\in P(\tour)$ such that its unique out-neighbor $x_a$ (which exists because $n\ge 2$) has $N^-(x)\ne \{a,a'\}$ for some other $a'\in P(\tour)$, then we can apply \eqref{eq:weak upper bound} to $u=a\in U$ to conclude $\dim(\tour,\sig)\le n-1$ since $\dim(\tour_u,\sig_u)=0$.  We can thus assume no such player exists.
	
	Now let $x$ be a match of $\tour$ with $|P(x)|$ as small as possible.  Note that we must have $N^-(x)\sub P(\tour)$, as any match in $N^-(x)$ would be a match with a strictly smaller player set by Proposition~\ref{Pu Sets}(v), a contradiction to the minimality of $|P(x)|$.  As such, $\tour_x$ is a single-elimination tournament with a single match $x$.  
	
	Observe that if $x$ has an out-neighbor $y\in N^+(x)$ then it can not have $N^-(y)=\{x,a\}$ for some player $a$, since such a player would mean we were in the case discussed above.  This means $x\in U$, and thus we can apply \eqref{eq:weak upper bound} to conclude
	\[\dim(\tour,\sig)\le n-|P(x)|+\dim(\tour_x,\sig_x)=n-1,\]
	with this last step using that $\tour_x$ is a tournament with a single match $x$ which can easily be checked to have $\dim(\tour_x,\sig_x)=|P(x)|-1$ for any choice of $\sig_x$.  We conclude the result.
\end{proof}
Our bound for standard single-elimination tournaments similarly follows from Theorem~\ref{upper bound technical}, though we will have to use a slightly different approach to get the stronger conclusion that there exists a single set $\c{B}$ which is $\sig$-resolving for every choice of $\sig$.

\begin{proof}[Proof of Theorem~\ref{standard}]
	The lower bound follows from Theorem~\ref{lower bound} applied with $x$ the sink.  It remains to prove for all $n\ge 2$ a power of 2 that there exists a set $\c{B}$ which is $\sig$-resolving for any choice of $\sig$ for the standard single-elimination tournament with $n$ players, which we denote by $\tour^n$ for ease of reference.  We prove this result by induction on $n$, the case $n=2$ holding with $\c{B}=\{B\}$ any choice of bracket.
	
	Assume that we have proven the result up to some $n\ge 4$ a power of 2.  Let $z$ denote the sink of $\tour^n$ and $u\in N^-(z)$ arbitrary.  Observe that $\tour_u^n$ is isomorphic to $\tour^{n-1}$, so inductively we know that there exists a set of brackets $\c{B}_u$ of size $n/4$ for $\tour_u^n$ which is $\sig_u$-resolving for every choice of $\sig_u$.  Let $\c{A}$ be the partition of $P(\tour^n)\sm P(u)$ guaranteed by Lemma~\ref{partitions exist}, noting that $|\c{A}|=n/4$ (as can be seen by observing that $|N^-(x)\cap P(\tour)\sm P(u)|$ is either 0 or 2 for every match $x\in M(\tour^n)\sm V(\tour^n_u)$).  It follows from Proposition~\ref{metric upper bound given partition} that there exists a set of brackets $\c{B}$ of size $|P(\tour^n)\sm P(\tour^n_u)|=n/2$ which is $\sig$-resolving whenever $\c{B}_u$ is $\sig_u$-resolving, which by hypothesis holds for every choice of $\sig$.  We conclude the result.
\end{proof}

\subsection{Resolving Number Upper Bound}
To prove our upper bound for resolving numbers, we need the following consequence of Lemma~\ref{favorite team}.

\begin{cor}\label{favorite player corollary}
	For any single-elimination bracket $\tour$ and scoring system $\sig$, we have for any bracket $\widehat{B}$ that the set $\widehat{\c{B}}=\{\widehat{B}_a: a\in P(\tour)\}$ is a $\sig$-resolving set.
\end{cor}
\begin{proof}
	The result is trivial if $\tour$ has a single vertex, so from now on we assume $\tour$ has at least 2 vertices. Assume for contradiction that there exist brackets $B,B'$ which have the same $\sig$-scores with respect to each element of $\widehat{\c{B}}$ and which have $B(x)\ne B'(x)$ for some match $x$, say with $a=B(x)$.  Note that the unique out-neighbor $x_a\in N^+(a)$ has at least one in-neighbor $v\ne a$ by definition of single-elimination tournaments and that $b:=\widehat{B}(v)$ does not equal $a$ by Proposition~\ref{Pu Sets}(v) and Lemma~\ref{bracket facts}(i).  It follows from Lemma~\ref{favorite team} and the fact that $\widehat{B}_a,\widehat{B}_b\in \widehat{\c{B}}$ that $B(x)=a$ implies $B'(x)=a$, giving the desired contradiction.
\end{proof}

We now prove our results for $\res(\tour,\sig)$.

\begin{proof}[Proof of Theorem~\ref{resolving with qmax alone}]
	The lower bound was established in Corollary~\ref{resolving qpair and qmax}, so  it remains for us to prove that if $\tour$ is a single-elimination tournament on at least 2 players and if 
	\[q_{\max}:=\max_{a\in P(\tour)} \Pr[R(z)=a],\]
	where $R$ is a uniformly bracket and $z$ is the sink of $\tour$, then $\res(\tour,\sig)\le (1-q_{max}) N$ for every scoring system $\sig$.  We prove this by a double counting argument.
	
	Let $\c{B}$ be a set of brackets which is not $\sig$-resolving and define $\c{P}$ to be the set of pairs of brackets $(\widehat{B},B)$ such that $B\in \{\widehat{B}_a:a\in P(\tour)\}$ and $B\notin \c{B}$.
	\begin{claim}
		For each bracket $\widehat{B}$, there exists a bracket $B$ with $(\widehat{B},B)\in \c{P}$.
	\end{claim}
	\begin{proof}
		If this were not the case then  $\{\widehat{B}_a:a\in P(\tour)\}$ would be entirely contained in $\c{B}$. This would imply that $\c{B}$ is $\sig$-resolving by Corollary~\ref{favorite player corollary}, a contradiction.
	\end{proof}
	\begin{claim}
		For each bracket $B$, the number of $\widehat{B}$ with $B\in\{\widehat{B}_a:a\in P(\tour)\}$ is at most $q_{\max}^{-1}$.
	\end{claim}
	\begin{proof}
		Let $b$ denote the player with $B(z)=b$.  Observe that $B\in \{\widehat{B}_a:a\in P(\tour)\}$ if and only if $\widehat{B}(x)=B(x)$ for all $x\in M(\tour)$ with $b\notin P(x)$.  As such, the number of $\widehat{B}$ with this property is exactly
		\[\prod_{x\in M(\tour): b\in P(x)} |N^-(x)|=\frac{\prod_{x\in M(\tour)} |N^-(x)|}{\prod_{x\in M(\tour): b\notin P(x)} |N^-(x)|}=\frac{N}{\prod_{x\in M(\tour): b\notin P(x)} |N^-(x)|}=\frac{1}{\Pr[R(z)=b]},\]
		with this last step using that a bracket $B'$ has $B'(z)=b$ if and only if $B'(x)=b$ for every $x$ with $b\in P(x)$ (and as such $B'$ can behave arbitrarily on all other $x$).  By definition this quantity is at most $q_{\max}^{-1}$, proving the result.
	\end{proof}
	These two claims together imply that
	\[N\le|\c{P}|\le q_{\max}^{-1} (N-|\c{B}|),\]
	or equivalently that $|\c{B}|\le N- q_{\max} N$, proving that every set $\c{B}$ which is not $\sig$-resolving has at most this size, giving the desired result.
\end{proof}

\section{Concluding Remarks}\label{sec:concluding}

We briefly mention three other natural variants of these problems that might be of interest for future study:
\begin{itemize}
	\item[(i)] What if instead of finding a set of brackets $\c{B}$ which distinguishes every pair of brackets from each other, we simply find a set of brackets which distinguishes a random bracket $R$ from every other bracket with high probability?
	\item[(ii)] What is the probability of a random set of brackets $\c{B}$ of a given size being $\sig$-resolving?
	\item[(iii)] What if our collection of brackets $\c{B}$ is chosen \textit{adaptively}, i.e.\ by first picking some $B_1$, and then choosing $B_2$ based on $\score_\sig(B_1,B)$, and so on.
\end{itemize}

We believe that we can to a large extent solve problems (ii) and (iii) and will do so in forthcoming work, so we focus our attention on (i) which can be stated more precisely as follows.

\begin{prob}
	Given a single-elimination tournament $\tour$ and a scoring system $\sig$, we say that a set of brackets $\c{B}$ $\sig$-resolves a bracket $B$ if every bracket $B'\ne B$ has $\score_\sig(B_i,B)\ne \score_\sig(B_i,B')$ for some $B_i\in \c{B}$.  Given an integer $t$ and some random bracket $R$ of $\tour$, estimate
	\[\max_{\c{B}_t} \Pr[\c{B}_t\ \sig\textrm{-resolves}\ R],\]
	where the maximum ranges over all possible sets of $t$ brackets $\c{B}_t$.
\end{prob}
We find this problem quite natural given that for March Madness we only care about distinguishing one particular bracket, namely the true one that actually happens.  Moreover, the distribution that this true bracket has in March Madness is typically not uniform, and it would be interesting to see if anything substantial could be said in the setting of not necessarily uniform random brackets.  For the case of uniform random brackets and $\tour$ being a standard single-elimination tournament with $n$ players, we can use Proposition~\ref{necessary condition} to show that
\[\max_{\c{B}_t} \Pr[\c{B}_t\ \sig\textrm{-resolves}\ R]\le 1-(1-2t/n)^2=4t/n-4t^2/n^2,\]
which is tight for $t=n/2$.
We can improve this for certain choices of $t$ and $\sig$.  Notably, it is not difficult to see that if $\score_\sig(B,B')$ can only take on $s$ possible values, then any set of $t$ brackets $\c{B}$ can $\sig$-distinguish at most $s^t$ brackets $B$.  In particular, if $\sig$ is the identically 1 scoring system for a standard single-elimination tournament then we have
\[\max_{\c{B}_t} \Pr[\c{B}_t\ \sig\textrm{-resolves}\ R]\le n^t 2^{1-n},\]
which gives a substantially stronger bound for small $t$.  It would be interesting to determine the true behavior of this probability as $t$ ranges over $1\le t\le n/2$.

\textbf{Acknowledgments}.  We thank Samuel Spirits~\cite{spirits} whose webnovel \textit{Executive Powers} about a super powered fighting tournament between Presidents of the United States served as the initial motivation for this paper.  We thank Aiya Kuchukova for recommending the random variant of this problem, Ruben Ascoli and Jade Lintott for providing some initial constructions, and Jason O'Neill for explaining to us how basketball works.


\bibliographystyle{abbrv}
\bibliography{BracketBIB}

\appendix 

\section{More on Resolving Numbers}
Here we prove a few more results around resolving numbers $\res(\tour,\sig)$.  In particular, our most general results are as follows where here $x_{a,b}$ is the match defined in Lemma~\ref{matches exist}.

\begin{thm}\label{appendix resolving}
	Let $\tour$ be a single-elimination tournament with at least 2 players.   Let $z$ denote the sink of $\tour$, let $R$ be a uniformly random bracket of $\tour$, and define the quantities
	\[q_{\max}=\max_{a\in P(\tour)}\Pr[R(z)=a],\]
	and 
	\[q_{\mathrm{pair}}=\min_{a,b\in P(\tour),\ a\ne b} \Pr[R(x_{a,b})\in \{a,b\}].\]
	If there are a total of $N$ brackets for $\tour$, then for every scoring system $\sig$ we have
	\[(1-q_{\mathrm{pair}})N< \res(\tour,\sig)\le (1-q_{\max})N.\]
	Moreover, for every $\tour$ with at least 2 players, there exist $\sig$ such that $\res(\tour,\sig)=(1-q_{\mathrm{pair}})N+1$.
\end{thm}

Observe that we have already proven the general lower and upper bounds of Theorem~\ref{appendix resolving} through Proposition~\ref{resolving lower bound technical} and Theorem~\ref{resolving with qmax alone}, so it remains only to prove the tightness of the lower bound.  As we did for Theorem~\ref{lower bound}, we do this for all $\sig$ with distinct subset sums.

\begin{lem}\label{tight resolving}
	If $\tour$ is a single-elimination tournament and if $\sig$ is a scoring system with distinct subset sums, then
	\[\res(\tour,\sig)=(1-q_{\mathrm{pair}})N+1.\]
\end{lem}
\begin{proof}
	The lower bound follows from Proposition~\ref{resolving lower bound technical}.  For the upper bound, let $\c{B}$ be a set of at least $(1-q_{\mathrm{pair}})N+1$ brackets.
	\begin{claim}
		For any distinct players $a,b$, there exists some $B_{a,b}\in \c{B}$ with $B_{a,b}(x_{a,b})\in \{a,b\}$.
	\end{claim}
	\begin{proof}
		Indeed, the number of brackets $B$ which have $B(x_{a,b})\in \{a,b\}$ for a given $a,b$ is at least $q_{\mathrm{pair}}N$ by definition of $q_{\mathrm{pair}}$, so our assumption on $|\c{B}|$ implies that at least one of these $q_{\mathrm{pair}}N$ brackets must be included in $\c{B}$, proving the claim.
	\end{proof}
	Now assume for contradiction that there exists brackets $B,B'$ with $\dist_\sig(B_i,B)= \dist_\sig(B_i,B')$ for all $B_i\in \c{B}$ and that $B(x)\ne B'(x)$.   for some $x$.  We choose such an $x$ with $|P(x)|$ as small as possible and let $a=B(x)$ and $b=B'(x)$.  Note that $x$ is necessarily a match because all brackets agree on $P(\tour)$ by definition.
	\begin{claim}
		We have $x=x_{a,b}$.
	\end{claim}
	\begin{proof}
		We certainly have $a,b\in P(x)$ by Lemma~\ref{bracket facts}(i) since $B(x)=a$ and $B'(x)=b$.  If there existed some $u\in N^-(x)$ with $a,b\in P(u)$ then Proposition~\ref{Pu Sets}(iii) and Lemma~\ref{bracket facts}(ii) imply that $B(u)=B(x)=a$ and $B'(u)=B'(x)=b$.  But this means $B(u)\ne B'(u)$ with $P(u)\subsetneq P(x)$ by Proposition~\ref{Pu Sets}(v), a contradiction to choosing $x$ so that $|P(x)|$ is as small as possible.  By Proposition~\ref{Pu Sets}(v) it must be that $x$ contains in-neighbors containing exactly one of $a,b$ in their player sets, proving that $x=x_{a,b}$.
	\end{proof}
	Without loss of generality assume $B_{a,b}\in \c{B}$ has $B_{a,b}(x_{a,b})=B_{a,b}(x)=a$.  This means that $B_{a,b}(x)=B(x)$ and $B_{a,b}(x)\ne B'(x)$.  But the fact that $\sig$ has distinct subset sums and $\dist_\sig(B_{a,b},B)=\dist_\sig(B_{a,b},B')$ implies that $B_{a,b}(x)=B(x)$ if and only if $B_{a,b}(x)=B'(x)$ by Lemma~\ref{distinct sums fact}, a contradiction.
\end{proof}

Finally, we note that the bounds $(1-q_{\max})N<\res(\tour,\sig)\le (1-2q_{\max})N$ of Theorem~\ref{resolving with qmax alone} are effective whenever $q_{\max}$ is far from $\half$.  When $q_{\max}\approx \half$ we can turn to the stronger lower bound of Theorem~\ref{appendix resolving} and still conclude that  $\res(\tour,\sig)$ is always at least a constant proportion of the total number of brackets for $\tour$ for every choice of $\tour$.

\begin{cor}\label{resolving quarter bound}
	If $\tour$ is a single-elimination tournament with at least 2 players and $N$ total brackets, then for any scoring system $\sig$ we have
	\[\res(\tour,\sig)> \quart N.\]
\end{cor}
The constant $\quart$ is best possible as there exists a 3-player tournament with $\res(\tour,\sig)=2=\quart N+1$.  One can improve this constant if one assumes a larger number of players, but the upper bound of Theorem~\ref{appendix resolving} implies that there exist tournaments with $\res(\tour,\sig)\le \half N$. 

\begin{proof}
	If $\tour$ has exactly two players then $N=2$ and one can check that $\res(\tour,\sig)=1> \quart N$ for any scoring system, so from now on we assume $\tour$ has some $n\ge 3$ players.  We will prove our result by showing that $q_{\mathrm{pair}}\le \frac{3}{4}$ in this case, from which the result will follow from Theorem~\ref{appendix resolving}.
	
	First assume that $\tour$ has a match $x$ with $P(\tour)\sub N^-(x)$.  It is not difficult to see that this is equivalent to saying that the only match of $\tour$ is $x$, which means that $x_{a,b}=x$ for all distinct $a,b$ and that every bracket is defined by picking one player to win $x$.  This in turn means that a uniformly random bracket $R$ simply picks one of the $n$ players uniformly at random to win this match.  We conclude that $q_{\mathrm{pair}}=2n^{-1}\le \frac{2}{3}\le \frac{3}{4}$, proving the result in this case.
	
	From now on we assume that $\tour$ has no match with $P(\tour)\sub N^-(x)$.  For each player $a$, let $x_a$ be the unique match with $x_a\in N^+(a)$.    By assumption of the case that we are in, there must exist distinct players $a,b$ with $x_a\ne x_b$.
	
	Let $R$ be a uniform random bracket, and for each player $c$ define $E_c$ to be the event that $R(x_c)=c$.  Observe that to have $R(x_{a,b})\in \{a,b\}$ we must have at least one of $E_a$ or $E_b$ occurring (otherwise $a$ and $b$ win no matches whatsoever), so $q_{\mathrm{pair}}\le \Pr[E_a\cup E_b]$.  Moreover, because $x_a\ne x_b$, the events $E_a,E_b$ are independent (since a uniform random bracket uniformly and independently picks $R(x)$ from $\{R(u):u\in N^-(x)\}$), and hence  
	\[q_{\mathrm{pair}}\le \Pr[E_a\cup E_b]=1-\Pr[\overline{E}_a\cap \overline{E}_b]=1-\Pr[\overline{E}_a]\Pr[\overline{E}_b]\le 1-\half\cdot \half, \]
	with this last step using that every match of a single-elimination tournament has at least two in-neighbors by definition, and hence $\Pr[E_c]\le \half $ for all $c$. We conclude the result.
\end{proof}

\end{document}